\documentclass[11pt,reqno]{amsart} 

\usepackage[utf8]{inputenc}
\usepackage[a4paper]{geometry}

\usepackage{amsmath}
\usepackage{amsfonts}
\usepackage{amssymb}
\usepackage{amsthm}
\usepackage{paralist}

\usepackage{natbib}
\usepackage{color}
\usepackage[colorlinks=true,linkcolor=blue,citecolor=blue,pdfborder={0 0 0}]{hyperref}

\usepackage{dsfont}
\usepackage{bm}

\usepackage{graphicx}
\usepackage{enumerate}

\newtheorem{theorem}{Theorem}[section]

\newtheorem{lemma}[theorem]{Lemma}
\newtheorem{corollary}[theorem]{Corollary}

\theoremstyle{definition}

\newtheorem{condition}[theorem]{Condition}

\theoremstyle{remark}

\numberwithin{equation}{section}

\newcommand{\dto}{\rightsquigarrow}

\newcommand{\R}{\mathbb{R}}

\newcommand{\Z}{\mathbb{Z}}
\newcommand{\N}{\mathbb{N}}
\newcommand\Prob{\mathbb{P}}    

\newcommand{\ip}[1]{\lfloor #1 \rfloor}

\newcommand{\GG}{\mathbb{G}}
\newcommand{\PP}{\mathbb{P}}

\newcommand{\Nc}{\mathcal{N}}
\newcommand{\Fc}{\mathcal{F}}

\newcommand{\scs}{\scriptscriptstyle}

\newcommand{\abs}[1]{\left\vert{#1}\right\vert}

\newcommand{\ind}{\operatorname{\bf{1}}}
\newcommand{\diff}{{\,\mathrm{d}}}

\newcommand{\e}{\mathrm{e}}

\newcommand{\Exp}{\operatorname{E}}
\newcommand{\Var}{\operatorname{Var}}
\newcommand{\Cov}{\operatorname{Cov}}

\newcommand{\RL}{\operatorname{RL}}

\newcommand{\argmax}{\operatornamewithlimits{\arg\max}}

\allowdisplaybreaks[4]

\begin{document}

\title[Sliding Blocks Estimators]{Inference for heavy tailed stationary time series based on sliding blocks}

\author{Axel B\"ucher}
\address{Ruhr-Universit\"at Bochum, Fakult\"at f\"ur Mathematik, Universit\"atsstr.\ 150, 44780 Bochum, Germany}
\email{axel.buecher@rub.de}

\author{Johan Segers}
\address{Universit\'{e} catholique de Louvain, Institut de Statistique, Biostatistique et Sciences Actuarielles, Voie du Roman Pays~20, B-1348 Louvain-la-Neuve, Belgium}
\email{johan.segers@uclouvain.be}

\date{\today}

\begin{abstract}
The block maxima method in extreme value theory consists of fitting an extreme value distribution to a sample of block maxima extracted from a time series. Traditionally, the maxima are taken over disjoint blocks of observations. Alternatively, the blocks can be chosen to slide through the observation period, yielding a larger number of overlapping blocks. Inference based on sliding blocks is found to be more efficient than inference based on disjoint blocks. The asymptotic variance of the maximum likelihood estimator of the Fr\'echet shape parameter is reduced by more than 18\%. Interestingly, the amount of the efficiency gain is the same whatever the serial dependence of the underlying time series: as for disjoint blocks, the asymptotic distribution depends on the serial dependence only through the sequence of scaling constants. The findings are illustrated by simulation experiments and are applied to the estimation of high return levels of the daily log-returns of the Standard \& Poor's~500 stock market index. \medskip

\end{abstract}

\maketitle

\section{Introduction}

Two major paradigms in extreme value theory are the block maxima method and the peaks-over-threshold method. The former, more traditional one consists of fitting an extreme value distribution to a sample of block maxima extracted from a (perhaps latent) underlying sample. The latter method consists of fitting a generalized Pareto distribution to the excesses in a sample over a high threshold.

Although the peaks-over-threshold method has become the standard one, there has been a renewed interest recently in the block maxima method, and more specifically in its asymptotic properties. The set-up is that of a triangular array of block maxima extracted from a stationary time series. The block size, $r$, tends to infinity as the sample size, $n$, tends to infinity, in such a way that the number of (disjoint) blocks, approximately $n/r$, tends to infinity as well. The underlying sequence of random variables can be independent \citep{FerDeh15, Dom15, DomFer17} or can exhibit serial dependence \citep{BucSeg14, BucSeg18}.

Usually, maxima are taken over disjoint blocks of observations. For instance, for a sequence of daily observations, one may extract weakly, monthly, quarterly or yearly maxima, see, e.g., \cite{Mcn98} and \cite{Lon00} or \cite{KatParNav02} for applications in finance or hydrology, respectively. Alternatively, one can slide a block or window of a given size through the sample and consider the corresponding maxima. Obviously, the blocks will be overlapping and thus dependent, even if the underlying sequence of random variables is independent. Still, as soon the underlying sequence is stationary, then so are the sliding block maxima. Moreover, the sample of sliding block maxima carries more information than the sample of disjoint block maxima, which suggests the possibility of more accurate inference. \cite{RobSegFer09}, \cite{Nor15} and \cite{BerBuc16} applied this idea to the estimation of the extremal index, a summary measure for the strength of serial dependence between extremes. They found that estimators based on sliding blocks were indeed more efficient than their counterparts based on disjoint blocks.

Here, we investigate the potential benefits of using maxima over sliding blocks rather than over disjoint blocks for fitting extreme value distributions. More precisely, we seek the asymptotic distribution of the maximum (quasi-)likelihood estimator for the shape and scale parameters of a Fr\'echet distribution. The likelihood is computed as if the sliding block maxima are independent, although they are not, as blocks may overlap.

The solution is based on Theorem~2.5 in \cite{BucSeg18}, which states high-level conditions for the consistency and asymptotic normality of the maximum likelihood estimator of the Fr\'echet parameter vector based on a general triangular array of dependent random variables. The biggest challenge is the computation of the estimator's asymptotic covariance matrix. In the course of the computations, we find new formulas for moments of pairs of jointly max-stable random variables in terms of their Pickands dependence function. These formulas are then applied to the bivariate Marshall--Olkin distribution, which describes the joint asymptotic distribution of the two maxima over a pair of overlapping blocks. The Marshall--Olkin parameter is a function of the proportion of overlap between the two blocks.

We find that the maximum likelihood estimator based on sliding blocks is more efficient than the maximum likelihood estimator based on disjoint blocks. For the estimator of the Fr\'echet shape parameter, the reduction in asymptotic variance is more than 18\%. Remarkably, this number does not depend on the serial dependence of the underlying stationary time series, in accordance to the findings for disjoint blocks in \cite{BucSeg18}. Moreover, the efficiency gain carries over to the estimation of high return levels. The asymptotic results are confirmed in numerical experiments. We illustrate the method by estimating high quantiles of quarterly maxima of daily log-returns of the S\&P500 index. The Monte Carlo simulations reveal another benefit of using sliding blocks: it makes the estimator more stable as a function of the block size.

The maximum likelihood estimator is defined and its asymptotic distribution is stated in Section~\ref{sec:main}. Estimation of high return levels is considered in Section~\ref{sec:return}, both theoretically and through a case study, followed by the results of a Monte Carlo simulation experiment in Section~\ref{sec:simul}. The proofs are given in Section~\ref{sec:proofs} while the covariance calculations are deferred to Appendix~\ref{app:cov}.

\section{Inference based on sliding blocks}
\label{sec:main}

\subsection{The maximum likelihood estimator based on sliding block maxima}

Let $(X_t)_{t\in \Z}$ be a strictly stationary time series: for any $k\in\N$ and for any $h, t_1, \dots, t_k\in \Z$, the distribution of $(X_{t_1+h}, \dots, X_{t_k+h})$ is the same as the distribution of $(X_{t_1}, \dots, X_{t_k})$. Further, let $P_{\alpha, \sigma}$ denote the Fr\'echet distribution with parameter $\theta = (\alpha, \sigma)' \in (0, \infty)^2$, given by its distribution function $P_{\alpha, \sigma}([0, x]) = \exp\{ -(x/\sigma)^{-\alpha} \}$ for $x > 0$. We assume the following maximum domain-of-attraction condition. The arrow $\dto$ denotes convergence in distribution. 

\begin{condition}[Max-domain of attraction] \label{cond:max}
For some $\alpha_0 \in (0,\infty)$, there exists a sequence $(\sigma_r)_{r\in\N}$, regularly varying at infinity with index $1/\alpha_0$, such that
\[
\max(X_1, \dots, X_r) / \sigma_r \dto P_{\alpha_0,1}, \qquad r \to \infty.
\]
\end{condition}

Regular variation of the sequence $(\sigma_r)_r$ with index $1/\alpha_0$ means that $\lim_{r \to \infty} \sigma_{\scs \lfloor rs \rfloor} / \sigma_r = s^{1/\alpha_0}$ for every $s > 0$, where $\lfloor \, \cdot \, \rfloor$ denotes the integer part.
A sufficient condition is that the common univariate distribution of the variables $X_t$ is in the max-domain of attraction of the Fr\'echet distribution (see Section~\ref{subsec:iid}) and that the extremal index $\theta$ of the time series $(X_t)_t$ exists and is positive \citep{leadbetter:1983}.

Suppose we observe a finite stretch of the time series, $X_1, \dots, X_n$. For some integer $r\in \{1, \dots, n\}$, let
\[
  M_{r,t} = M_{t:(t+r-1)} = \max\{ X_{t}, \dots, X_{t+r-1} \}, \qquad t = 1, \dots, n-r+1,
\]
denote the maximum over the $r$ successive observations starting at time point $t$. The sequence $M_{r,1}, \dots, M_{r,k}$ with $k=n-r+1$ is referred to as the sequence of  \textit{sliding block maxima}. In contrast, the classical block maxima method in extreme value statistics is based on the sequence of \textit{disjoint block maxima} $M_{r,1}, M_{r,r+1}, \dots, M_{r, (m-1)r+1}$, where $m=\lfloor n/r \rfloor$. The common big blocks/small blocks heuristics suggests that the latter sequence may be regarded as asymptotically independent and Fr\'echet distributed. Any sensible estimator within the statistical model $\mathcal P= \{ P_\theta^{\otimes m}: \theta=(\alpha, \sigma)' \in (0,\infty)^2\}$  is hence a sensible estimator when applied to the sample of disjoint block maxima as well.

Unfortunately, this idea cannot be directly transferred to the sample of sliding block maxima, as that sequence is certainly not asymptotically independent, not even for an underlying iid time series. Still, the sample of sliding block maxima is stationary and the asymptotic distribution of a single such block maximum is Fr\'echet. We can therefore estimate the Fr\'echet parameters by moment matching, for instance. Being based on empirical moments only, the maximum likelihood estimator for independent sampling from the Fr\'echet distribution (model $\mathcal P$) is a case in point.

Existence and uniqueness of the maximum likelihood estimator in model $\mathcal P$ is studied in Section~2.1 of \cite{BucSeg18}. The estimator is defined as
\begin{equation}
\label{eq:MQLE}
  (\hat \alpha_n, \hat \sigma_n)
  =
  \argmax_{ \theta \in (0, \infty)^2 } \sum_{t=1}^{n-r+1} \ell_\theta(X_{n,t}),
\end{equation}
where $\ell_\theta(x) = \log \diff \exp\{ - (x/\sigma)^{-\alpha} \} / \diff x$ is the contribution of an observation at $x > 0$ to the Fr\'echet log-likelihood of the parameter vector $\theta$ and where
\begin{equation}
\label{eq:Xnt}
X_{n,t} = M_{r,t} \vee c
\end{equation}
with an arbitrary truncation constant $c>0$. The reason for the left-truncation is that otherwise some of the block maxima could be zero or negative. Asymptotically, the truncation constant does not matter thanks to Condition~\ref{cond:div} below. In practice, it should be chosen as small as possible, e.g., equal to the square root of the machine precision.  
By Lemma~2.1 in \cite{BucSeg18}, the maximizer exists and is unique as soon as the $n-r+1$ values $X_{n,1}, \ldots, X_{n,n-r+1}$ are not all equal, which our conditions will guarantee to occur with probability tending to one. Note that the likelihood is constructed as if the sliding block maxima were independent, although they are not, not even if the underlying sequence is independent, since some blocks overlap. Therefore, the estimator may be more accurately referred to as a maximum quasi-likelihood estimator.

\subsection{Asymptotic normality}

For our main result, we need a couple of additional conditions.  First of all, let $r=r_n$ be an integer sequence tending to infinity such that $r_n=o(n)$.
The next condition is a slight adaptation of Condition 3.2 in \cite{BucSeg18}.

\begin{condition}[All block maxima of size $\ip{r_n/2}$ diverge] \label{cond:div}
For every $c>0$, the probability of the event that all disjoint block maxima of size $\tilde r_n=\ip{r_n/2}$ are larger than $c$ converges to 1.
\end{condition}

The condition in fact implies that the probability of the event that all sliding block maxima of size $r_n$ are larger than $c$ converges to 1 as well. It therefore guarantees that the left-truncation above does not matter asymptotically. In Section~\ref{subsec:iid}, the condition will be shown to hold for iid time series, provided the block sizes are not too small.

The following three conditions are Conditions~3.3, 3.4 and 3.5 in \cite{BucSeg18}. The alpha-mixing coefficients of the sequence $(X_t)_t$ are defined as
\[
  \alpha(k) 
  = 
  \sup \{ 
    \lvert \Prob(A \cap B) - \Prob(A) \Prob(B) \rvert : 
    A \in \sigma(X_j, j \le i), \, B \in \sigma(X_{j+k}, j \ge i), \, i \in \Z 
  \},
\]
for $k = 1, 2, \ldots$.

\begin{condition}[$\alpha$-Mixing with rate] \label{cond:alpha}
We have $\lim_{\ell \to \infty}\alpha(\ell) = 0$. Moreover, there exists $\omega>0$ such that 
\[
\lim_{n\to\infty} (n/r_n)^{1+\omega} \alpha(r_n) = 0.
\]
\end{condition}

\begin{condition}[Moments] \label{cond:moment}
There exists $\nu>2/\omega$ with $\omega$ from Condition~\ref{cond:alpha} such that
\begin{equation}
\label{eq:moment}
\limsup_{r\to\infty} \Exp\big[ g_{\nu, \alpha_0} \big( (M_{r,1} \vee 1) / \sigma_r \big) \big] < \infty,
\end{equation}
where $g_{\nu,\alpha_0}(x) = \{ x^{-\alpha_0} \ind(x \le \e) + \log(x) \ind(x > \e) \}^{2+\nu}$.
\end{condition}

An elementary argument shows that if Condition~\ref{cond:moment} holds, then \eqref{eq:moment} continues to hold with $M_{r,1} \vee 1$ replaced by $M_{r,1}\vee c$, for arbitrary $c > 0$.

Let $P=P_{\alpha_0,1}$ and write $Pf= \int_0^\infty f(x) \diff P(x)$ for a real-valued function on $(0,\infty)$.  Further, let
\begin{align} \label{eq:fij}
f_1(x) = x^{-\alpha_0} \log(x), \qquad  f_2(x) = x^{-\alpha_0}, \qquad f_3(x) = \log(x)
\end{align}
and note that, by Lemma B.1 in \cite{BucSeg18},
\[
  Pf_1 = -\alpha_0^{-1} \Gamma'(2) = \alpha_0^{-1} (\gamma - 1), \qquad 
  Pf_2 = \Gamma(2) = 1, \qquad 
  Pf_3 = -\alpha_0^{-1} \Gamma'(1) = \alpha_0^{-1} \gamma,
\]
where $\Gamma$  and $\Gamma'$ denote the gamma function $\Gamma(z) = \int_0^\infty t^{z-1} \e^{-t} \diff t$ and its derivative, respectively, while $\gamma = 0.5772\ldots$ denotes the Euler--Mascheroni constant. 

\begin{condition}[Bias] \label{cond:bias}
There exists $c_0>0$ such that for $j=1,2,3$,
\[
\lim_{n\to\infty} \sqrt{\frac{n}{r_n}} \Big(\Exp\big[ f_j \big( (M_{r,1} \vee c_0) / \sigma_r \big) \big] - Pf_j  \Big) = B(f_j)
\]
with $f_j$ as defined in \eqref{eq:fij}.
\end{condition}

As an interesting consequence of Theorem~\ref{theo:weak} below, the limit $B(f_j)$ can be seen not to depend on the constant $c_0$ under the conditions of the theorem.

Finally, we define the empirical process based on the left-truncated sliding block maxima in \eqref{eq:Xnt} as
\[
\GG_n f = \sqrt{\frac{n}{r_n}} \bigg(k_n^{-1} \sum_{t=1}^{k_n} f(X_{n,t}/\sigma_{r_n}) - Pf \bigg).
\]
Recall the maximum (quasi-)likelihood estimator in \eqref{eq:MQLE}. The following theorem is the main result of this paper.

\begin{theorem} \label{theo:weak}
Suppose Conditions~\ref{cond:max}--\ref{cond:bias} are met. 
Then, for any $c>0$ and with probability tending to one, there exists a unique maximizer $(\hat{\alpha}_n, \hat{\sigma}_n)$ of the Fr\'echet log-likelihood based on the left-truncated sliding block maxima $X_{n,1}, \dots, X_{n,k_n}$, see \eqref{eq:MQLE}, and we have, as $n \to \infty$,
\begin{align*}
   \sqrt{m_n} \,
  \begin{pmatrix} 
    \hat{\alpha}_n - \alpha_0  \\ 
    \hat{\sigma}_n / \sigma_{r_n} - 1  
  \end{pmatrix} 
  &= 
  M(\alpha_0)
  \begin{pmatrix}
    \GG_n x^{-\alpha_0} \log(x) \\
    \GG_n x^{-\alpha_0} \\
    \GG_n \log(x)
  \end{pmatrix} + o_\Prob(1) 
  \dto
  \mathcal{N}_2 \bigl( M(\alpha_0) \, B, \; \bm \Sigma(\alpha_0) \bigr).
\end{align*}
Here,  $m_n= n/r_n$, $B=(B(f_1), B(f_2), B(f_3))'$ as in Condition~\ref{cond:bias} and
\begin{align*}
  M(\alpha_0)
  &=
 \frac{6}{\pi^2} 
 \begin{pmatrix}
    \alpha_0^2 & \alpha_0(1-\gamma) & - \alpha_0^2 \\
    \gamma-1 & -(\Gamma''(2) +1) / \alpha_0 &  1-\gamma
  \end{pmatrix}, \\
\bm \Sigma(\alpha_0) &= M(\alpha_0) \bm \Sigma_{\bm Y} M(\alpha_0)^T \approx
  \left(\begin{array}{ll}
\phantom{-}0.4946 \, \alpha_0^2 & -0.3236  \\
 -0.3236 & \phantom{-}0.9578 \, \alpha_0^{-2}
 \end{array}
 \right), 
\end{align*}
with $\bm \Sigma_{\bm Y}=(\sigma_{ij})_{i,j=1}^{3}$ as in Corollary~\ref{cor:cov} and equation~\eqref{eq:cov}. 
\end{theorem}

It is interesting to note that the limiting covariance matrix is substantially smaller than for the estimator based on disjoint blocks, see Theorem 3.6 in \cite{BucSeg18}: 
\begin{equation*}
  I_{(\alpha_0,1)}^{-1} 
  = 
  \frac{6}{\pi^2} 
  \begin{pmatrix}
    \alpha_0^2 & (\gamma-1) \\ 
    (\gamma-1) &  \alpha_0^{-2} \{ (1-\gamma)^2 +\pi^2/6 \} 
  \end{pmatrix}
  \approx
  \left(\begin{array}{ll}
    \phantom{-}0.6080 \, \alpha_0^2 & -0.2570 \\ 
    -0.2570 &  \phantom{-}1.1087 \, \alpha_0^{-2}
  \end{array}\right).
\end{equation*}
The improvement is independent of the value of $\alpha_0$ and of the serial dependence of the time series (e.g., of any of the characteristics like the extremal index or the cluster distribution). In particular, the quotient of the asymptotic variances for the shape and scale parameters are $0.8135$ and $0.8639$, respectively. 

More generally, the delta method implies that the asymptotic distribution of a tail quantity that can be written as a smooth function of $(\hat{\alpha}_n, \hat{\sigma}_n)'$ will be normal with asymptotic variance equal to $\bm \beta' I_{(\alpha_0, 1)}^{-1} \bm \beta$ (disjoint blocks) or $\bm \beta' \bm\Sigma(\alpha_0) \bm \beta$ (sliding blocks), where $\bm \beta$ is a $2 \times 1$ vector of partial derivatives depending on the estimator. The ratio of asymptotic variances is thus equal to
\begin{equation}
\label{eq:ratio}
  \frac%
    {\bm \beta' \bm\Sigma(\alpha_0) \bm \beta}%
    {\bm \beta' I_{(\alpha_0, 1)}^{-1} \bm \beta}.
\end{equation}
By the method of Lagrange multipliers, this ratio can be found to attain its minimum and maximum at the smallest and largest eigenvalues of the matrix $\bm\Sigma(\alpha_0) (I_{\scs (\alpha_0, 1)}^{-1})^{-1} = \bm\Sigma(\alpha_0) I_{(\alpha_0, 1)}$. Independently of $\alpha_0$, these eigenvalues are equal to $0.9413$ and $0.6448$, respectively, and provide precise lower and upper bounds to the ratio in \eqref{eq:ratio}.

Should the underlying time series be positive, one might be tempted to omit the left-truncation introduced above (which amounts to setting $c=0$). Working out the asymptotic theory is possible, but at the cost of more complicated conditions. Even if $X_t > 0$ almost surely, the lower tail of the random variable $X_t$ in a neighbourhood of zero must not be too heavy for the moment and bias conditions to be true without truncation. In practice, left-truncation at a suitable small constant appears to be non-restrictive anyway whence we do not pursue this issue any further.

\subsection{Sliding block maxima extracted from an iid sequences}
\label{subsec:iid}

If the sequence $(X_t)_t$ is iid, then all conditions can be expressed in terms of the univariate, marginal distribution function $F(x) = \Prob(X_t \le x)$. Condition~\ref{cond:max} is equivalent to regular variation of the function $- \log F$ at infinity with index $-\alpha_0$, that is,
\begin{equation}
\label{eq:RV}
  \lim_{u \to \infty} \frac{- \log F(ux)}{- \log F(u) } = x^{-\alpha_0},
  \qquad x \in (0, \infty).
\end{equation}
The scaling sequence can be chosen as $\sigma_r = \inf \{ u \ge 1 : F(u) \ge \e^{-1/r} \}$, for $r = 1, 2, \ldots$.

To control the bias (Condition~\ref{cond:bias}), we need to reinforce regular variation in~\eqref{eq:RV} to second-order regular variation of the function $-\log F$ together with a growth restriction on the block size sequence $(r_n)_n$. The following condition is identical to Condition~4.1 in \citet[Section~4]{BucSeg18}; see also Remark~4.3 therein for additional context. For $\tau \in \R$, define $h_\tau : (0, \infty) \to \R$ by
\begin{equation*}
  h_\tau(x) 
  = 
  \int_1^x y^{\tau - 1} \, \diff y
  =
  \begin{cases}
    \dfrac{x^\tau - 1}{\tau}, & \text{if $\tau \neq 0$,} \\[1ex]
    \log(x), & \text{if $\tau = 0$.}
  \end{cases}
\end{equation*}

\begin{condition}[Second-Order Condition]
\label{cond:secor}
There exists $\alpha_0 \in (0, \infty)$, $\rho \in (-\infty, 0]$, and a real function $A$ on $(0, \infty)$ of constant, non-zero sign such that $\lim_{u \to \infty} A(u) = 0$ and such that, for all $x \in (0, \infty)$,
\begin{equation}
\label{eq:SV:2}
  \lim_{u \to \infty} 
  \frac{1}{A(u)} 
  \left( 
    \frac{-\log F(ux)}{-\log F(u)} - x^{-\alpha_0} 
  \right) 
  = 
  x^{-\alpha_0} \, h_\rho(x).
\end{equation}
\end{condition}

Let $\psi = \Gamma' / \Gamma$ denote the digamma function. For $(\alpha_0, \rho) \in (0, \infty) \times (-\infty, 0]$, define the bias function
\begin{equation}
\label{eq:bias:iid}
  B(\alpha_0, \rho) 
  = 
  - \frac{6}{\pi^2} 
  \begin{pmatrix} 
    b_1(\abs{\rho}/\alpha_0) \\ 
    b_2(\abs{\rho}/\alpha_0) / \alpha_0^2
  \end{pmatrix},
\end{equation}
where
\begin{equation*}
  b_1(x) = 
  \begin{cases}
    (1+x) \, \Gamma ( x ) \{ \gamma  + \psi(1+x) \}, & \text{if $x > 0$,} \\
    \dfrac{\pi^2}{6}, & \text{if $x = 0$,}
  \end{cases}
\end{equation*}
and
\begin{equation*}
  b_2(x) = 
  \begin{cases} 
    -\dfrac{\pi^2}{6x} + 
      (1+x) \, \Gamma ( x )  
      \{ \Gamma''(2) + \gamma + (\gamma-1) \, \psi(1 + x)\},
    & \text{if $x > 0$,} \\
    0, & \text{if $x = 0$}.
  \end{cases}
\end{equation*} 
The graphs of these functions are depicted in Figure~1 in \cite{BucSeg18}.

\begin{theorem}
\label{theo:ml}
Let $X_1, X_2, \ldots$ be independent random variables with common distribution function $F$ satisfying Condition~\ref{cond:secor}.
Let the block sizes $r_n$ be such that $r_n \to \infty$ and $m_n = \lfloor n / r_n \rfloor \to \infty$ as $n \to \infty$ and assume that 
\begin{equation}
\label{eq:ka}
  \lim_{n\to \infty} \sqrt{m_n} \, A(a_{r_n}) = \lambda \in \R.
\end{equation}
Then, for any $c>0$ and with probability tending to one, there exists a unique maximizer $(\hat{\alpha}_n, \hat{\sigma}_n)$ of the Fr\'echet log-likelihood based on the left-truncated sliding block maxima $X_{n,1}, \dots, X_{n,k_n}$, and we have
\begin{equation*}
  \sqrt{m_n} 
  \begin{pmatrix}
    \hat{\alpha}_n - \alpha_0 \\
    \hat{\sigma}_n / \sigma_{r_n} - 1
  \end{pmatrix}
  \dto 
  \Nc_2 \left( \lambda \, B(\alpha_0, \rho), \, \bm \Sigma(\alpha_0) \right),
  \qquad n \to \infty,
\end{equation*}
with $\bm \Sigma(\alpha_0)$ as in Theorem~\ref{theo:weak} and $B( \alpha_0, \rho )$ as in \eqref{eq:bias:iid}.
\end{theorem}

Compared to the estimator based on disjoint block maxima \citep[Theorem~4.2]{BucSeg18}, the asymptotic bias is the same, but the asymptotic variance is smaller, as explained after Theorem~\ref{theo:weak}.

\section{Application to return level estimation}
\label{sec:return}

\subsection{Estimator}

Let $F_r(x)=\Prob(M_{1:r} \le x)$. For $T \ge 1$, the $T$-return level of the sequence of disjoint block maxima is defined as the $1-1/T$ quantile of $F_r$, 
that is,
\[
  \RL(T,r) = F_r^{\leftarrow}(1-1/T) = \inf \{ x \in \R : F_r(x) \ge 1 - 1/T \}.
\]
Since disjoint blocks are asymptotically independent, it will take on average $T$ disjoint blocks of size $r$ until the first such block whose maximum exceeds $\RL(T, r)$.

By Condition~\ref{cond:max}, for large $r$, we may approximate $F_r$ by $G_{\alpha_0, \sigma_r}$,
 the cdf of the Fr\'echet distribution with shape parameter $\alpha_0$ and scale parameter $\sigma_r$. The quantile function of the Fr\'echet family is given by
$
G_{\alpha, \sigma}^{-1}(p) = \sigma\{-\log(p)\}^{-1/\alpha_0}.
$
A reasonable estimator of $\RL(T,r_n)$ is therefore
\[
\widehat \RL_n(T,r_n) = \hat \sigma_n b_T^{-1/\hat \alpha_n}, \qquad b_T=- \log(1-1/T).
\]
Also, let $\widetilde\RL(T,r_n) = \sigma_{r_n} b_T^{-1/\alpha_0}.$

\begin{corollary}
\label{cor:RT}
Additionally to the conditions in Theorem~\ref{theo:weak} assume the bias condition
\[
\Lambda_n(T) = \sqrt{m_n}  \left( \frac{\widetilde \RL(T,r_n) }{\RL(T,r_n) } -1 \right) \to \Lambda(T) \in \R, \qquad n \to \infty.
\]
Then, as $n\to\infty$,
\begin{align*}
\sqrt{m_n} \left( \frac{\widehat \RL_n(T,r_n) }{\RL(T,r_n) } -1 \right) 
&=
\beta(T,\alpha_0)' \cdot \sqrt{m_n} \,
  \begin{pmatrix} 
    \hat{\alpha}_n - \alpha_0  \\ 
    \hat{\sigma}_n / \sigma_{r_n} - 1  
  \end{pmatrix} 
  + \Lambda_n(T) + o_\Prob(1) \\
&\dto  \Nc\Big(\beta(T, \alpha_0) ' M(\alpha_0) \, B + \Lambda(T)\,,\,   \beta(T, \alpha_0) ' \bm \Sigma(\alpha_0) \beta(T, \alpha_0)\Big)
\end{align*}
where $\beta(T, \alpha_0) = \bigl( \alpha_0^{-2} \log(b_T), 1 \bigr)'$.
\end{corollary}

The same result holds true for the disjoint blocks estimator, but with $\bm \Sigma(\alpha_0)$ replaced by $I_{(\alpha_0,1)}^{-1}$, the inverse of the Fisher information matrix for the Fr\'echet family. In Table~\ref{tab:var1}, the asymptotic variances are given for various values of $T$ and with $\alpha_0=1$.

\begin{table}[t!]
\centering
\begin{tabular}{rrrrrrrr}
  \hline
  \hline
 T & 50 & 100 & 200 & 500 & 1000 & 5000 & 10000 \\ 
  \hline
Sliding Blocks & 11.01 & 14.40 & 18.26 & 24.07 & 29.02 & 42.35 & 48.87 \\ 
  Disjoint Blocks & 12.37 & 16.34 & 20.88 & 27.77 & 33.66 & 49.59 & 57.41 \\ 
  Ratio & 0.89 & 0.88 & 0.87 & 0.87 & 0.86 & 0.85 & 0.85 \\ 
   \hline
   \hline
\end{tabular}

\medskip
\caption{\label{tab:var1} Asymptotic variances for the sliding and disjoint blocks versions of $\widehat \RL_n(T, r_n)$, alongside with their ratio, for $\alpha_0=1$.}
\end{table}

\subsection{Case study}

We consider daily log-returns on the S\&P500 index (data downloaded from Yahoo Finance) in the fifty-year period from 1967 to 2016, yielding $n=12\,584$ observations in total. For such a long period, the log-returns can hardly be modelled by a stationary time series, and therefore we consider ten-year periods instead, yielding approximately $n = 2\,500$ daily observations per period. We consider a block size equal to $r = 62$, which corresponds to approximately one quarter, yielding $m = 40$ disjoint quarters in a ten-year period. In light of the results in \cite{Mcn98} and \cite{Lon00}, such a block size guarantees that the block maxima are approximately Fr\'echet distributed, while at the same time the effective sample size of $m\approx 40$ is large enough to guarantee a reasonably small variance of the estimators. We estimate the $1-1/T$ quantile of the distribution of the quarterly maximum using the sliding block maxima and we check whether this level is exceeded by the quarterly maximum immediately following the ten-year training sample. We then let the ten-year window roll through the whole fifty-year period in steps of one quarter, giving $40 \times 4  = 160$ estimated return levels in total, of which, on average, $160/T$ should be exceeded. We consider $T = 20, 40, 80$, for which one would thus expect $8, 4, 2$ exceedances, respectively. Doing the calculations separately for the positive and negative log-returns, we find $7, 3, 1$ exceedances for the wins and $10,7,1$ exceedances for the losses. 

The results are shown in Figure~\ref{fig:SP500}. The black lines depict the quarters following each ten-year training period, from 1977:Q1 up to 2016:Q4. The estimated return levels are represented by colored lines; for instance, the value of the red line at 1977:Q1 corresponds to the estimated $1-1/20$ quantile of the cdf of the quarterly block maxima, estimated from the data in the period 1967--1976. The dots correspond to exceedances of the estimated return levels.

In Figure~\ref{fig:SP500-Shape-Scale}, we show estimated values and $95\%$ confidence intervals (based on the normal approximation and ignoring the bias) for the Fr\'echet shape and scale parameters. The black lines are the same as in Figure~\ref{fig:SP500}, except for an affine transformation. The impact of the occurrence of large block maxima is clearly visible (decrease in $\alpha$, increase in $\sigma$).

\begin{figure}[t!]
\begin{center}
\centerline{\includegraphics[width=1\textwidth]{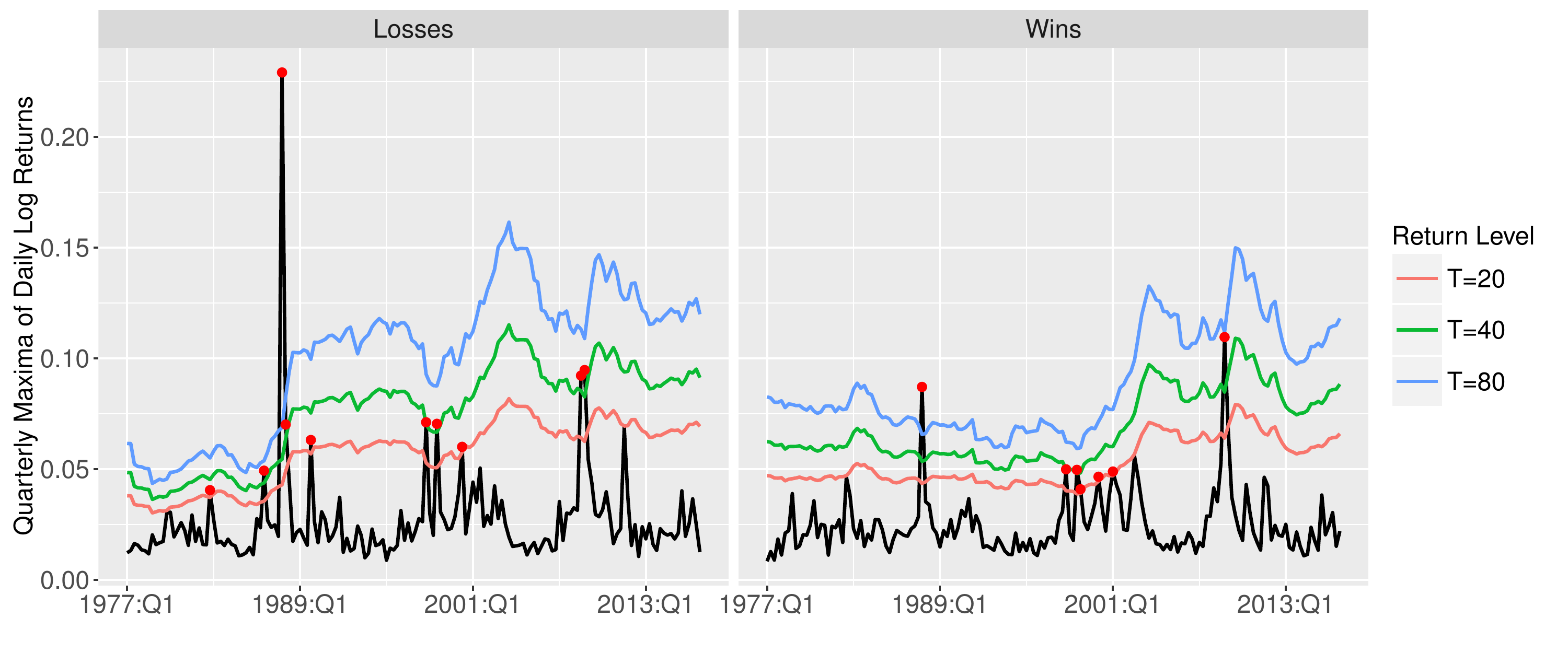}}
\vspace{-.5cm}
\caption{\label{fig:SP500} Quarterly block maxima of the negative (left) and positive (right) log-returns of the S\&P500 index (black lines) together with estimates of the $1-1/T$ quantile of the block maximum distribution (coloured lines). The estimates for a given quarter are based on sliding block maxima in the ten-year period immediately preceding that quarter.}
\end{center}
\end{figure}

\begin{figure}[t!]
\begin{center}
\centerline{\includegraphics[width=1\textwidth]{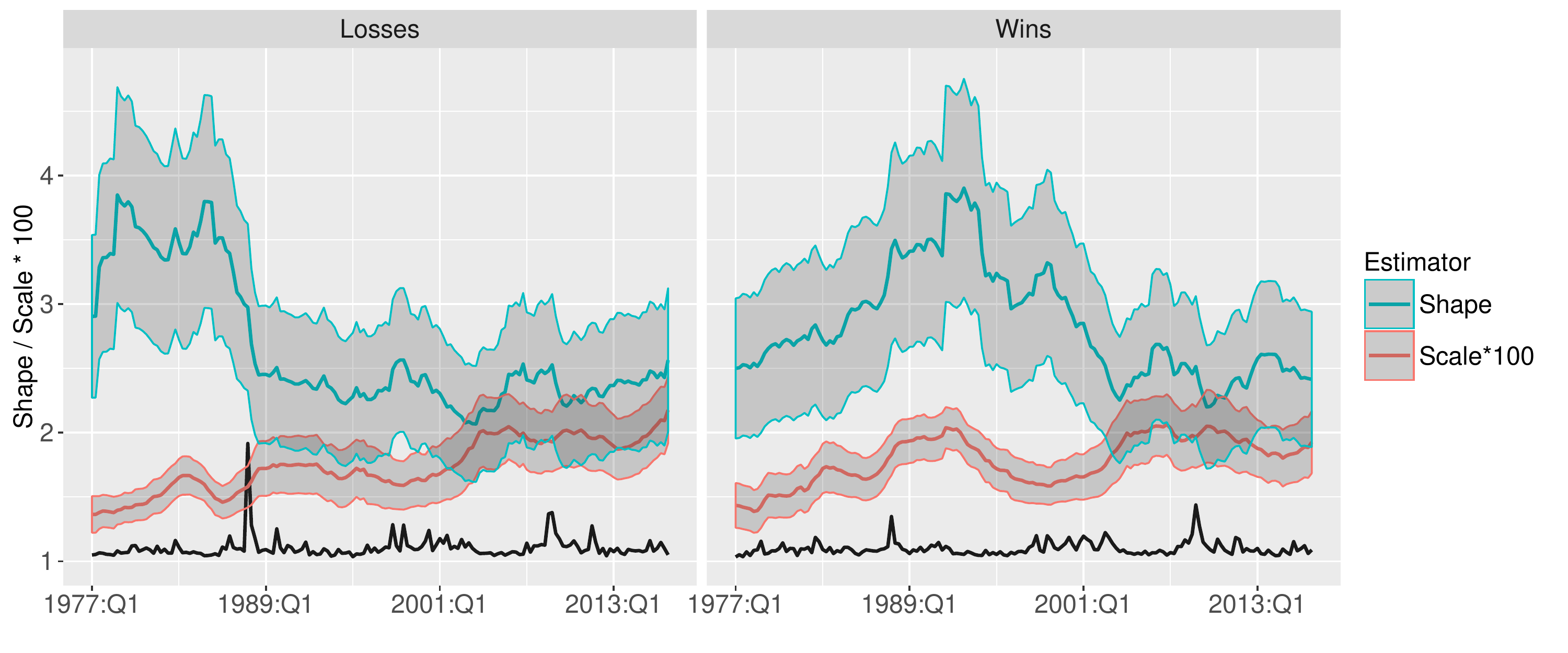}}
\vspace{-.5cm}
\caption{\label{fig:SP500-Shape-Scale} Estimated Fr\'echet shape (blue) and scale (red) parameters with pointwise $95\%$ confidence intervals of the distribution of the quarterly block maxima of the negative (left) and positive (right) log-returns of the S\&P500 index. The estimates for a given quarter are based on the sliding block maxima in the ten-year period immediately preceding that quarter and are depicted together with the actually observed block maximum (affinely transformed, black line).}
\end{center}
\end{figure}

\section{Simulation Study}
\label{sec:simul}

We compare the performance of the disjoint and sliding blocks variations of the maximum likelihood estimator of the Fr\'echet shape parameter $\alpha$. We also compare the results with the Hill estimator \citep{Hil75}, which is the maximum likelihood estimator for the one-parameter Pareto distribution given by $\Pr(Y > y) = (y \vee 1)^{-\alpha}$ fitted to the relative excesses over a high threshold. To put the estimators on equal footing, we set the threshold equal to the $m+1$ largest order statistic, of which there are $m$ excesses, where $m = \ip{n/r}$ is the number of disjoint blocks of size $r$ and which acts as an effective sample size.

We consider two different dependence scenarios: iid sequences and a max-autoregressive (ARMAX) model
\[
  X_t = \max \{ \beta X_{t-1}, (1-\beta) Z_t \}, \qquad t \in \Z,
\]
with parameter $\beta \in [0, 1]$ 
and where $(Z_t)_t$ are nonnegative iid random variables whose distribution is in the max-domain of attraction of the Fr\'echet distribution with shape parameter $\alpha > 0$. The process $(X_t)_t$ is strictly stationary and admits the causal representation $X_t = \max_{j \ge 0} \{ \beta^j (1-\beta) Z_{t-j} \}$. Since $\Prob(M_{1:r} \le x) = \Prob(X_1 \le x) \, \Prob \{ Z_1 \le x/(1-\beta) \}^{r-1}$, rescaled block maxima of $X_t$ converge weakly to the same Fr\'echet distribution, but with scaling sequence depending on $\beta$. For the simulations, we fix $\beta=1/2$, and use a burn-in period of length $200$ to arrive at approximate stationarity.

We consider three different choices for the iid case and for the innovation distribution of the max-autoregressive model: the Fr\'echet distribution itself, the Pareto distribution with shape parameter $\alpha$, and the absolute value of the Student \emph{t}-distribution with $\alpha$ degrees of freedom. 

The mean squared error, squared bias and variance  for the estimators of $\alpha_0$ are shown in Figure~\ref{fig:MSE-shape}. We consider a fixed sample size $n=1\,000$, block sizes $r=2,3, \dots, 50$ for the blocks estimators (i.e., $m=n/r$ ranging from $20$ to 500) and numbers of upper order statistics $m=20,30, \dots, 500$ for the Hill estimator. The results are based on $3\,000$ repetitions.

The bias-variance trade-off is clearly visible: small $m$ (large $r$) yields a large variance but a small bias, while increasing $m$ (decreasing $r$) decreases the variance but potentially increases the bias (with some exceptions for the Hill estimator in the ARMAX-model). The appearance of the bias curves in the iid scenarios may be explained as follows: The Hill estimator is the maximum likelihood estimator in the iid Pareto model, while the blocks estimators are maximum (quasi-)likelihood estimators in the iid Fr\'echet models. The absolute $t$-distribution is in between. The size of the bias is ordered accordingly. 

In all cases, the sliding blocks estimator is more accurate than the disjoint blocks estimator due to its smaller variance.

Finally, another advantage of the sliding blocks estimator is that its trajectories fluctuate less as a function of $r$, as illustrated in Figure~\ref{fig:Hill4}.

\begin{figure}[t!]
\begin{center}
\centerline{\includegraphics[width=\textwidth]{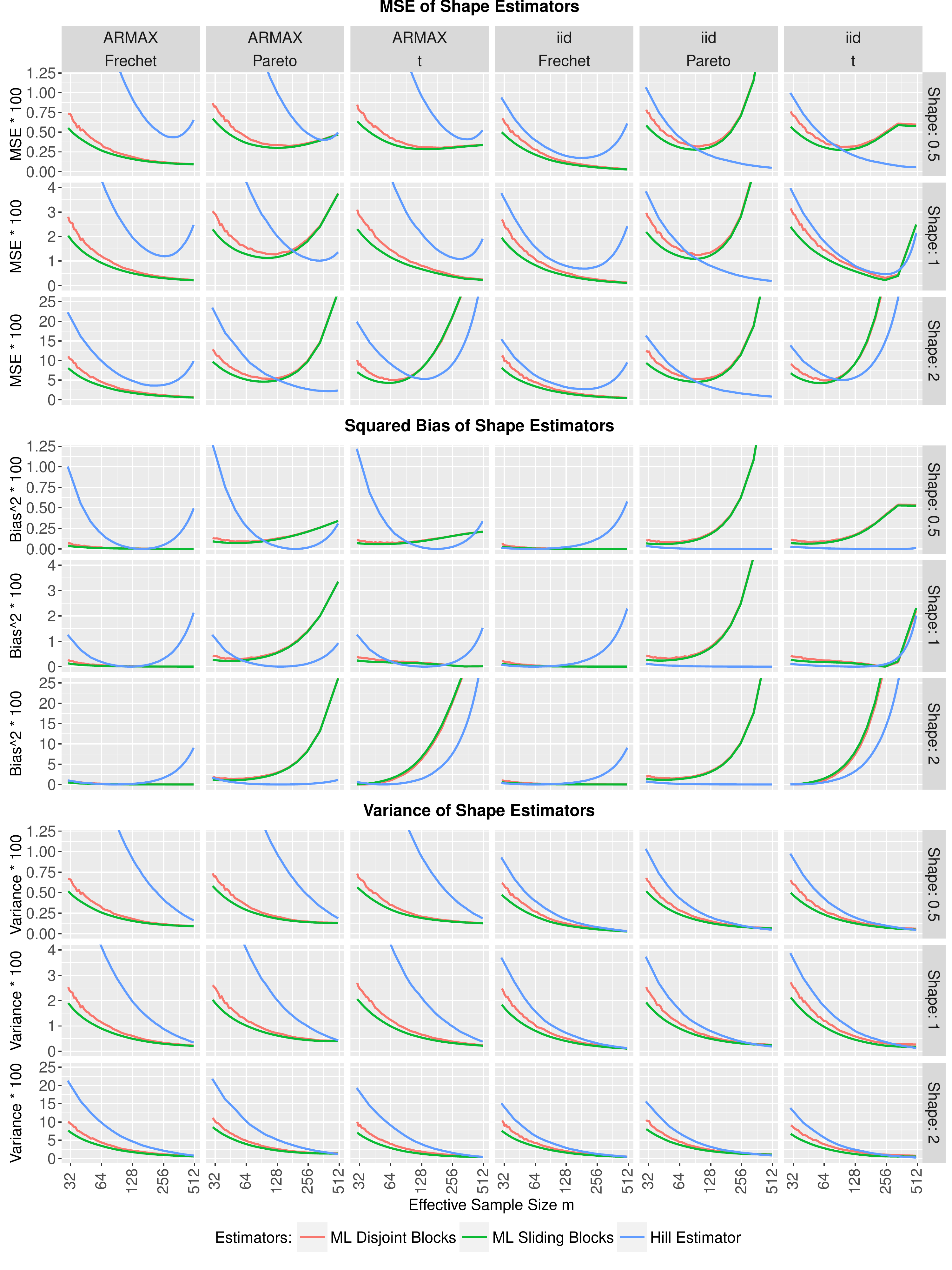}}
\vspace{-.5cm} 
\caption{\label{fig:MSE-shape} 
MSE, Squared Bias and Variance for the estimation of $\alpha_0$ (multiplied by 100) as a function of the effective sample size $m$, where $m=n/r$ for the blocks estimators and $m$ is equal to the number of upper order statistics used for the Hill estimator.
}
\end{center}
\vspace{-2cm}  
\end{figure}

\begin{figure}[t!]
\begin{center}
\vspace{-.1cm}
\centerline{\includegraphics[width=1.07\textwidth]{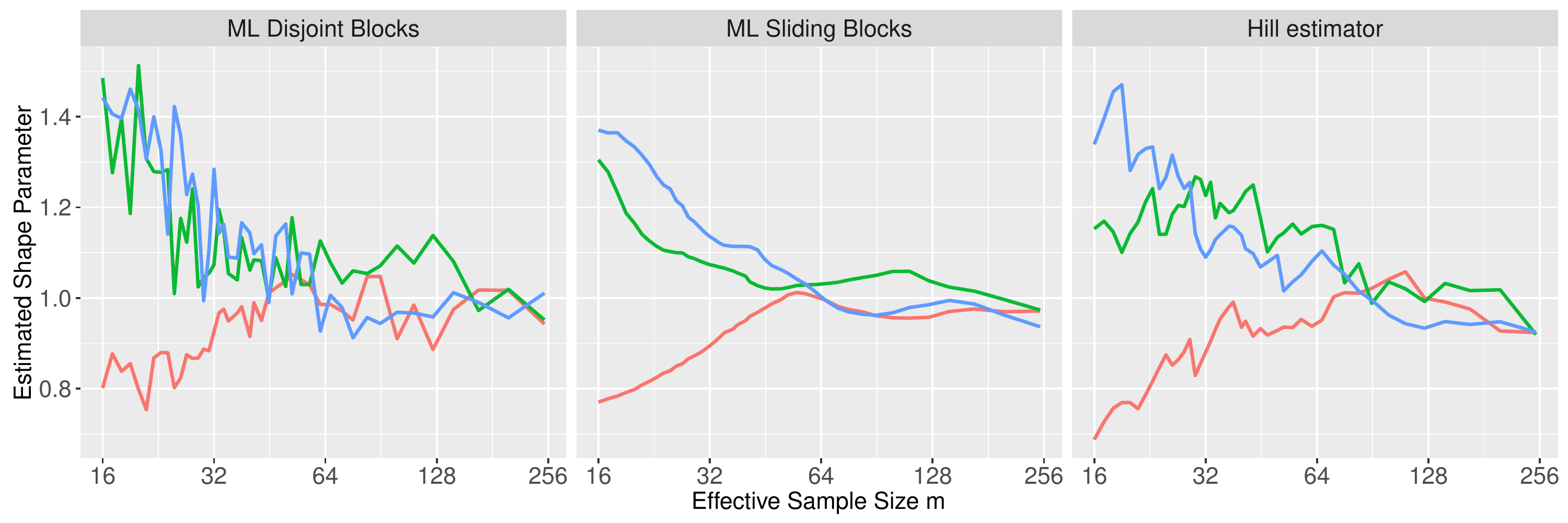}}
\vspace{-.3cm} 
\caption{\label{fig:Hill4} Three typical trajectories of three different estimators of the Fr\'echet shape parameter, based on iid samples of size $n=1\,000$ from the $t$-distribution with 1 degree of freedom (i.e., $\alpha_0=1$). The blocks estimators are based on the 44 unique values $r$ that are obtained by calculating the integer part of $n/m_0$ with $m_0$ ranging from $16, 17, \dots$ up to $250$. The Hill estimator is based on the respective 44 values $m=\ip{n/r}$. 
}
\vspace{-.5cm}
\end{center}
\end{figure}

\section{Proofs and auxiliary results}
\label{sec:proofs}

The proof of Theorem~\ref{theo:weak} is based on a sequence of auxiliary lemmas. 
Let $Z_{r_n,t}= (M_{r_n,t} \vee c) / \sigma_{r_n}$, with $c$ specified in the subsequent lemmas.
All convergences will be for $n\to\infty$, if not stated otherwise. Throughout the proofs, we will write $r=r_n$, $m=m_n$ etc. 

\begin{lemma}[Joint weak convergence of sliding block maxima] \label{lem:joint}
Suppose that Condition~\ref{cond:max} is met and that there exists an integer sequence $(\ell_n)_n$ such that $\ell_n=o(r_n)$ and $\alpha(\ell_n) = o(\ell_n/r_n)$ as $n\to\infty$. Then, for any $c\ge0$ and any $\xi \ge 0$, we have, for $x > 0$ and $y > 0$ and as $n \to \infty$,
\begin{multline}
\label{eq:Galphaxi}
  \lim_{r \to \infty}
  \Prob( Z_{r_n,1} \le x, Z_{r_n, 1 + \ip{r_n \xi}} \le y ) \\
  = G_{\alpha_0, \xi}(x, y) =
  \begin{cases}
    \exp \left\{
      - \xi x^{-\alpha_0} 
      - (1-\xi) (x \wedge y)^{-\alpha_0} 
      - \xi y^{-\alpha_0}
    \right\}, &
    \text{if $0 \le \xi \le 1$,} \\[1ex]
    \exp( - x^{-\alpha_0} - y^{-\alpha_0} ), &
    \text{if $\xi \ge 1$.}
  \end{cases}
\end{multline}
\end{lemma}

Surprisingly, the joint limiting law does not depend on any quantities related to the serial dependence of the time series (like the extremal index). The two margins of the limit distribution are always $P_{\alpha_0, 1}$. For $\xi \ge 1$, the two block maxima are asymptotically independent. For $0 \le \xi \le 1$, the limiting distribution $G_{\alpha_0, \xi}$ is bivariate max-stable with $P_{\alpha_0, 1}$ margins and with Pickands dependence function \citep{Pic81} equal to
\begin{equation}
\label{eq:MO:A}
  A_\xi(w) 
  = \xi + (1-\xi)(w \wedge (1-w)), \qquad w \in [0, 1].
\end{equation}
If $(Z_1, Z_2) \sim G_{\alpha_0, \xi}$, then $(Z_1^{-\alpha_0}, Z_2^{-\alpha_0})$ is a pair of unit exponential random variables with joint survival function
\begin{align}
\nonumber
  \Prob( Z_1^{-\alpha_0} \ge s, Z_2^{-\alpha_0} \ge t )
  &= \exp \{ - \xi s - (1-\xi) (s \wedge t) - \xi t \} \\
\label{eq:MO}
  &= \exp \{ - (s+t) \, A_{\xi}( s/(s+t) ) \}.
\end{align}
This is the Marshall--Olkin distribution \citep{MarOlk67} with dependence parameter $\xi$. The distribution depends on $\xi$ in such a way that the dependence increases as $\xi$ decreases, i.e., as the overlap between the two blocks increases.

\begin{proof}[Proof of Lemma~\ref{lem:joint}]
Write $r=r_n$, $\ell=\ell_n$ and $\alpha=\alpha_0$.
Since $c/\sigma_r\to 0$ as $r\to\infty$, we may redefine $Z_{r,i}=M_{r,i}/\sigma_r$.
By Condition~\ref{cond:max},
\[
  \lim_{r \to \infty} \sigma_{\ip{r\xi}}/\sigma_r 
   = \xi^{1/\alpha}
\]
for any $\xi > 0$. As a consequence, 
\[
  \Prob(M_{1:\ip{r\xi}} \le \sigma_r x) 
  = \Prob \bigl\{ M_{1:\ip{r\xi}} \le  \sigma_{\ip{r\xi}} (\tfrac{\sigma_r}{\sigma_{\ip{r\xi}}} x ) \bigr\}
  \to \exp(-\xi x^{-\alpha}), \qquad r\to \infty.
\]
Consider the case $\xi \in (0,1)$.
We will show below that
\begin{align} \label{eq:jointz}
&\ \Prob( Z_{r,1} \le x, Z_{r,1+\ip{r\xi}} \le y)  \nonumber \\
=&\ 
\Prob \bigl\{ M_{1:\ip{r\xi}} \le \sigma_r x, \, M_{(\ip{r\xi}+1):r} \le  \sigma_r(x \wedge y), \, M_{(r+1):(r+\ip{r\xi})} \le \sigma_ry \bigr\} \nonumber \\
=&\
\Prob \bigl\{ M_{1:\ip{r\xi}} \le \sigma_rx \bigr\} \,
\Prob \bigl\{ M_{(\ip{r\xi}+1):r} \le \sigma_r(x \wedge y) \bigr\} \,
\Prob \bigl\{ M_{(r+1):(r+\ip{r\xi})} \le \sigma_ry \bigr\} + o(1)
\end{align}
as $r\to \infty$. As a consequence of the previous two displays and by strict stationarity, we obtain \eqref{eq:Galphaxi} for $\xi \in (0, 1)$. The case $\xi \ge 1$ can be treated similarly while the case $\xi = 0$ is trivial. It remains to show \eqref{eq:jointz}.

As a consequence of Lemma~A.8 in \cite{BucSeg18}, we have
\[
\lim_{n\to\infty} \Prob(M_{1:(r'-\ell)} < M_{(r'-\ell+1):r'}) =0
\]
for any sequence $r'$ such that $r'/r$ is bounded away from 0 and infinity; note that Condition~3.1 in that paper follows from our assumption that $(\sigma_r)_r$ is regularly varying.
Applying this result four times and using the fact that $\lim_{\ell\to\infty}\alpha(\ell)=0$, we may write the expression in the middle line of \eqref{eq:jointz} as
\begin{align*} 
&\
\Prob \{ 
  M_{1:(\ip{r\xi}-\ell)} \le \sigma_r x, \, 
  M_{(\ip{r\xi}+1):(r-\ell)} \le  \sigma_r(x \wedge y), \, 
  M_{(r+1):(r+\ip{r\xi})} \le \sigma_ry
\}  + o(1) \nonumber \\
=&\ 
\Prob \{ M_{1:(\ip{r\xi}-\ell)} \le \sigma_r x \} \, 
\Prob \{ M_{(\ip{r\xi}+1):(r-\ell)} \le  \sigma_r(x \wedge y) \} \,
\Prob \{ M_{(r+1):(r+\ip{r\xi})} \le \sigma_ry \} + o(1)  \\
=&\ 
\Prob \{ M_{1:\ip{r\xi}} \le \sigma_r x \} \,
\Prob \{ M_{(\ip{r\xi}+1):r} \le  \sigma_r(x \wedge y) \} \,
\Prob \{ M_{(r+1):(r+\ip{r\xi})} \le \sigma_ry \} + o(1) 
\end{align*}
as $n\to\infty$, which proves \eqref{eq:jointz}.
\end{proof}

\begin{lemma}[Asymptotic covariances of functions of sliding block maxima]
\label{lem:mom1}
Suppose Conditions~\ref{cond:max} and \ref{cond:moment} are met and that there exists an integer sequence $(\ell_n)_n$ such that $\ell_n=o(r_n)$ and $\alpha(\ell_n) = o(\ell_n/r_n)$ as $n\to\infty$. Then, for any $c>0$, $\xi \in [0,1]$ and any pair of measurable functions  $f, g$ on $(0,\infty)$ which are continuous almost everywhere and satisfy
\[
(|f| \vee |g|)^2 \le g_{\eta,\alpha_0}(x) = \{ x^{-\alpha_0}\ind(x\le e) + \log(x) \ind(x> e) \}^{2+\eta}
\]
for some $0<\eta<\nu$, we have
\[
  \lim_{n\to\infty} 
  \Cov \bigl( f(Z_{r_n,1}), \, g(Z_{r_n, 1+\ip{r_n\xi}}) \bigr)
  =
  \Cov_{\alpha_0, \xi} \bigl( f(Z_1), g(Z_2) \bigr)
\]
where the right-hand side means that $(Z_1, Z_2) \sim G_{\alpha_0, \xi}$ as in \eqref{eq:Galphaxi}.
\end{lemma}

\begin{proof}[Proof of Lemma~\ref{lem:mom1}] 
The result is a simple consequence of Lemma~\ref{lem:joint}, the Cauchy--Schwarz inequality and Example~2.21 in \cite{Van98}.
\end{proof}

\begin{lemma}\label{lem:cov}
Suppose Conditions~\ref{cond:max}, \ref{cond:alpha} and \ref{cond:moment} are met. Then, for any pair of measurable functions  $f,g$ which are continuous almost everywhere and satisfy
\[
(|f| \vee |g|)^2 \le g_{\eta,\alpha_0}(x) = \{ x^{-\alpha_0}\ind(x\le \mathrm{e}) + \log(x) \ind(x > \mathrm{e}) \}^{2+\eta}
\]
for some $0<\eta<\nu$, we have
\[
  \sigma_{f,g}=\lim_{n \to \infty} \Cov( \GG_n f, \GG_n g)
  =
  2 \int_0^1 
    \Cov_{\alpha_0,\xi} \bigl( f(Z_1), g(Z_2) \bigr) \, 
  \diff \xi.
\]
\end{lemma}

\begin{proof}[Proof of Lemma~\ref{lem:cov}] 
Let $\ell_n= \max\{s_n, \ip{r_n \sqrt{\alpha(s_n)}}\}$, where $s_n=\ip{\sqrt{r_n}}$.  Then $\ell_n\to \infty$, $\ell_n=o(r_n)$ and  $\alpha(\ell_n) = o(\ell_n/r_n)$ as $n\to\infty$, so that the results of Lemma~\ref{lem:joint}--\ref{lem:mom1} become available.

For $h=1, \dots, \ip{n/r_n}$, let $I_h = \{(h-1)r_n+1, \dots,h r_n \}$ denote the set of indices making up the $h$th disjoint block of size $r_n$. For simplicity assume that $m_n=n/r_n$ is an integer. Then, we may write
\[
\frac{1}{k_n} \sum_{t=1}^{k_n} f(Z_{r,t})= \frac{1}{k_n}\sum_{h=1}^{m_n} A_h, \qquad 
\frac{1}{k_n} \sum_{t=1}^{k_n} g(Z_{r,t})= \frac{1}{k_n}\sum_{h=1}^{m_n} B_h,
\]
where $A_h = \sum_{s\in I_h} f(Z_{r,s})$ and  $B_h = \sum_{s\in I_h} g(Z_{r,s})$.
As a consequence, 
\begin{align} \label{eq:covdec}
&\hspace{-.5cm}\Cov( \GG_n f, \GG_n g) \nonumber \\
&= \frac{m_n}{k_n^2} \Big( m_n \Cov(A_1, B_1) + \sum_{h=1}^{m_n-1} (m_n-h)\bigl\{ \Cov(A_1, B_{1+h}) +   \Cov(B_1, A_{1+h})\big \} \Big) \nonumber\\
&= \frac{m_n^2}{k_n^2} \Cov(A_2, B_1+ B_2+B_3)  - \frac{m_n}{k_n^2}\Cov(A_2, B_1+B_3)  \nonumber \\
& \hspace{3cm} + \frac{m_n^2}{k_n^2} \sum_{h=2}^{m_n-1} (1-\tfrac{h}{m_n}) \bigl\{ \Cov(A_1, B_{1+h}) +   \Cov(B_1, A_{1+h}) \big \}.
\end{align}
Let us proceed by showing that 
\begin{align} \label{eq:sn}
\lim_{n \to \infty}S_n = 2\int_0^1 \Cov_{{\alpha_0,\xi}} (f(Z_1), g(Z_2)) \, \diff \xi, \quad \text{ where } S_n =\frac{1}{r_n^2} \Cov(A_2, B_1+ B_2+B_3).
\end{align}
For that purpose, define functions $g_{n1}$ and $g_{n2}$ on the positive real line by
\[
g_{n1}(\xi) =  \Cov \bigl( f(Z_{r,1}), g(Z_{r,1+\ip{r \xi}}) \bigr), \qquad
g_{n2}(\xi) =  \Cov \bigl( f(Z_{r,1+\ip{r \xi}}), g(Z_{r,1}) \bigr).
\]
We may then write
\begin{align*}
\frac{1}{r_n^2} \Cov(A_2, B_2) 
&= 
\frac{1}{r_n^2} \sum_{s=1}^{r_n} \sum_{t=1}^{r_n}  \Cov \bigl( f(Z_{r,s}), g(Z_{r,t}) \bigr) \\
&= 
\frac{1}{r_n} g_{n1}(0) + \frac{1}{r_n} \sum_{h=1}^{r_n-1} \big(1-\tfrac{h}{r_n}\big) \Big\{ g_{n1}\big( \tfrac{h}{r_n}\big)  + g_{n2}\big(\tfrac{h}{r_n}\big)  \Big\}, \\
\frac{1}{r_n^2} \Cov(A_2, B_3) 
&= 
\frac{1}{r_n^2} \sum_{s=1}^{r_n} \sum_{t=r_n+1}^{2r_n}  \Cov \bigl( f(Z_{r,s}), g(Z_{r,t}) \bigr) \\
&= 
\frac{1}{r_n} \sum_{h=1}^{r_n-1} \tfrac{h}{r_n} g_{n1}\big( \tfrac{h}{r_n}\big)   + \frac{1}{r_n} \sum_{h=r_n}^{2r_n-1} (2-\tfrac{h}{r_n}) g_{n1}\big( \tfrac{h}{r_n}\big), \\
\frac{1}{r_n^2} \Cov(A_2, B_1) 
&= 
\frac{1}{r_n^2} \sum_{s=1}^{r_n} \sum_{t=r_n+1}^{2r_n}  \Cov \bigl( f(Z_{r,s}), g(Z_{r,t}) \bigr) \\
&= 
\frac{1}{r_n} \sum_{h=1}^{r_n-1} \tfrac{h}{r_n} g_{n2}\big( \tfrac{h}{r_n}\big)   + \frac{1}{r_n} \sum_{h=r_n}^{2r_n-1} (2-\tfrac{h}{r_n}) g_{n2}\big( \tfrac{h}{r_n}\big).
\end{align*}
As a consequence of the three previous formulas, 
\[
S_n = \int_0^1 \{ g_{n1}(\xi) + g_{n2}(\xi) \} \diff\xi + R_n 
\]
where the remainder $R_n$ satisfies
\[
|R_n| \le \frac{1}{r_n} \lvert g_{n1}(0) \rvert + 2 \int_1^2 \{ \lvert g_{n1}(\xi) \rvert + \lvert g_{n2}(\xi) \rvert \} \diff\xi.
\]
Note that $\lim_{n \to \infty } g_{n\ell}(\xi) = \Cov_{\alpha_0,\xi} (f(Z_1), g(Z_2))$ by Lemma~\ref{lem:mom1}, for $\ell=1,2$. In particular, the limit is zero for $\xi \ge 1$. 
Hence, we obtain both $R_n \to 0$ and \eqref{eq:sn} by dominated convergence.

Let us finally show that the sum on the right-hand side of \eqref{eq:covdec} is negligible. The lemma then follows from  \eqref{eq:covdec} and \eqref{eq:sn}. For that purpose, start by considering  the sum over those summands for which $h\ge 3$. 
In this case, the observations making up $A_1$ and $B_{1+h}$ are separated by  $r_n(h-2)$ observations. As a consequence, invoking Lemma~3.11 in~\cite{DehPhi02}, we have
\[
\Cov(A_1, B_{1+h}) \le 10 \, \lVert A_1 \rVert_{2+\nu} \, \lVert B_1 \rVert_{2+\nu} \, \{ \alpha(r_n(h-2)) \}^{\nu/(2+\nu)},
\]
with $\nu>0$ from Condition~\ref{cond:moment}.
Since $(m_n/k_n)^2 \lVert A_1\rVert_{2+\nu} \lVert B_1 \rVert_{2+\nu}=O(1)$, we can bound the sum involving $\Cov(A_1, B_{1+h})$ by a multiple of 
$
\sum_{h=1}^{m_n-3} \{\alpha(h r_n)\}^{\nu/(2+\nu)},
$
which converges to 0 by Condition~\ref{cond:alpha}.
The same argument can be used to handle the sum involving $\Cov(B_1, A_{1+h})$ for $h\ge 3$. It remains to consider the summand corresponding to $h=2$.  We may write
\begin{align*}
\frac{1}{r_n^2} \Cov(A_1, B_3) 
&= 
\frac{1}{r_n^2} \sum_{s=1}^{r_n} \sum_{t=2r_n+1}^{3r_n}  \Cov(f(Z_{r,s}), g(Z_{r,t}) ) \\
&= 
\frac{1}{r_n} \sum_{h=r_n}^{2r_n-1} \tfrac{h}{r_n} g_{n1}\big( \tfrac{h}{r_n}\big)   + \frac{1}{r_n} \sum_{h=2r_n}^{3r_n-1} (2-\tfrac{h}{r_n}) g_{n1}\big( \tfrac{h}{r_n}\big),
\end{align*}
which can be bounded in absolute value by $3 \int_1^3 \lvert g_{n1}(\xi) \rvert \diff \xi$. As before, this integral converges to zero by dominated convergence and Lemma~\ref{lem:mom1}. Similarly, $\lvert \Cov(B_1, A_3) \rvert = o(r_n^2)$, and the proof is finished.
\end{proof}

Recall the polygamma function of order $m\ge 0$ and the Riemann zeta function: 
\[
 \psi^{(m)}(z) = \frac{d^{m+1}}{dz^{m+1}} \log \Gamma(z) , \quad(z>0), \quad
 \qquad \zeta(z)=\sum_{k=1}^\infty k^{-z}, 
 \quad (z>1).
\]
A special value that we will need is Ap\'ery's constant, $\zeta(3) \approx 1.2020569$. Recall the functions $f_1, f_2, f_3$ in \eqref{eq:fij}. 
For $i,j \in \{1,2,3\}$, define
\begin{equation}
\label{eq:sigmaij}
  \sigma_{ij} = 
  2\int_0^1 
    \Cov_{{\alpha_0,\xi}} \bigl( f_i(Z_1), \, f_j(Z_2) \bigr) \, 
  \diff \xi.
\end{equation}

\begin{corollary}
\label{cor:cov}
For $\sigma_{ij}$ as in \eqref{eq:sigmaij}, we have
\begin{align*}
\sigma_{11} &= \alpha_0^{-2} \Big[4 \log(2) \Big( \psi(2)^2 + \frac{\pi^2}{6} - \psi(2) \log(2) + \log^2(2)/3 \Big) + \psi(2) \frac{\pi^2}{3} - \frac{7}{8}\zeta(3) - 2 \psi(2)^2 \Big],  \\
\sigma_{22} &= 4\log(2)-2,  \\
\sigma_{33}&=  \alpha_0^{-2} \Big[8 \log(2) -4\Big], \\
\sigma_{12} &= \alpha_0^{-1} \Big[ 2\log^2(2)  - \frac{\pi^2}{6} - (1-\gamma)( 4\log(2) -2) \Big], \\
\sigma_{13}&=  \alpha_0^{-2}\Big[(1+\psi(2))\frac{\pi^2}{6}+2\log^2(2)-4\psi(2)\log(2)+2\psi(2)-\frac{7}{16}\zeta(3)\Big], \\
\sigma_{23} &= \alpha_0^{-1} \Big[ 4\log(2) - 2 - \frac{\pi^2}{6}  \Big].
\end{align*} 
\end{corollary}

\begin{proof}
If $(Z_1, Z_2) \sim G_{\alpha_0, \xi}$, then $(S, T) = (Z_1^{-\alpha_0}, Z_2^{-\alpha_0})$ is a pair of unit exponential random variables whose distribution is jointly min-stable with Marshall--Olkin Pickands dependence function $A_\xi$; see~\eqref{eq:MO:A}--\eqref{eq:MO}. Further, we have
\begin{align*}
  f_1(Z_1) 
  &= Z_1^{-\alpha_0} \log(Z_1) = -\alpha_0^{-1} S \log(S), \\
  f_2(Z_1)
  &= Z_1^{-\alpha_0} = S, \\
  f_3(Z_1)
  &= \log(Z_1) = -\alpha_0^{-1} \log(S),
\end{align*}
and similarly for $f_j(Z_2)$. 
We obtain
\begin{align*}
  \Cov_{\alpha_0,\xi} \bigl(f_1(Z_1), f_1(Z_2)\bigr)
  &= 
  \alpha_0^{-2} 
  \Cov_\xi \bigl( S \log(S), T \log(T) \bigr), 
\\
  \Cov_{\alpha_0,\xi} \bigl(f_2(Z_1), f_2(Z_2)\bigr)
  &= 
  \Cov_\xi (S, T), 
\\
  \Cov_{\alpha_0,\xi} \bigl(f_3(Z_1), f_3(Z_2)\bigr)
  &=
  \alpha_0^{-2} 
  \Cov_\xi \bigl( \log(S), \log(T) \bigr), 
\\
  \Cov_{\alpha_0,\xi} \bigl(f_1(Z_1), f_2(Z_2)\bigr)
  &=
  -\alpha_0^{-1}
  \Cov_\xi \bigl( S \log(S), T \bigr),
\\
  \Cov_{\alpha_0,\xi} \bigl(f_1(Z_1), f_3(Z_2)\bigr)
  &=
  \alpha_0^{2}
  \Cov_\xi \bigl( S \log(S), \log(T) \bigr),
\\
  \Cov_{\alpha_0,\xi} \bigl(f_2(Z_1), f_3(Z_2)\bigr)
  &=
  -\alpha_0^{-1}
  \Cov_\xi \bigl( S, \log(T) \bigr),
\end{align*}
where the index $\xi \in [0, 1]$ stresses the dependence of the covariances on the Marshall--Olkin parameter. In Appendix~\ref{app:cov}, the covariances on the right-hand side are denoted by $H_{k,\ell}(a, b; \xi) = \Cov_\xi( S^a (\log(S))^k, T^b (\log(T))^\ell)$, and their integrals over $\xi \in [0, 1]$ are computed in Corollary~\ref{cor:cov:int}. Multiplying by two, we find the stated formulas.
\end{proof}

Approximate values of $\sigma_{ij}$ in \eqref{eq:sigmaij} are summarized in the following matrix:
\begin{equation}
\label{eq:cov}
\bm \Sigma_{\bm Y}=(\sigma_{ij})_{i,j=1}^{3} = \frac{1}{\alpha_0^2} 
\left(
\begin{array}{lll}
  \phantom{-}1.5140 & -1.0107 \cdot \alpha_0 & \phantom{-}0.8712 \\[0.5ex]
 -1.0107 \cdot \alpha_0 & \phantom{-} 0.7726\cdot \alpha_0^2 & -0.8723 \cdot\alpha_0 \\[0.5ex]
 \phantom{-}0.8712 & -0.8723 \cdot \alpha_0  & \phantom{-}1.5434
 \end{array}
 \right).
\end{equation}

\begin{proof}[Proof of Theorem~\ref{theo:weak}]
The theorem is a consequence of Theorem~2.5  in \cite{BucSeg18}. In the notation of that paper, let $v_n=\sqrt{m_n}$ and recall that $k_n=n-r_n+1$ and $m_n=\ip{n/r_n}$. As before, write $k=k_n$, $m=m_n$, and so on. Subsequently, redefine $X_{n,t} = M_{r, t} \vee c_0$ with $c_0 > 0$ from Condition~\ref{cond:bias}. Recall that, as a consequence of Condition~\ref{cond:div}, such a redefinition does not change the estimator on a sequence of events whose probability converges to one. Hence, the asymptotic distribution is unaffected as well and in particular the asymptotic bias (which is identifiable from the limiting distribution) does not depend on $c_0$. We need to show the following three properties:
\begin{enumerate}[(i)]
\item  $\lim_{n\to\infty} \Prob(X_{n,1}= \dots = X_{n,k}) = 0$.

\item There exist constants $0<\alpha_- < \alpha_0 < \alpha_+ < \infty$ such that
\[
\PP_n f = \frac{1}{k} \sum_{t=1}^{k} f(X_{n,t}/\sigma_{r}) \dto \int_0^\infty f(x) \, p_{\alpha_0,1}(x) \, \diff x
\]
for all $f \in \Fc_2(\alpha_-, \alpha_+)$, where
\[
 \Fc_2(\alpha_-, \alpha_+)
= 
\{x \mapsto \log x\} \cup \{ x \mapsto x^{-\alpha} (\log x)^k : k=0,1,2, \; \alpha \in (\alpha_-,\alpha_+)\}.
\]

\item
For $f_j$ as in \eqref{eq:fij}, we have
\begin{equation*}
  (\GG_n f_1, \GG_n f_2, \, \GG_n f_3)^T
  \dto
  \bm{Y} \sim \Nc_3( B, \bm \Sigma_{\bm Y})
  , \qquad n \to \infty,
\end{equation*}
where $\GG_n$ and $B$ are as in Theorem~\ref{theo:weak} and where $\bm \Sigma_{\bm Y}$ is as in Corollary~\ref{lem:cov}, see in particular equation~\eqref{eq:cov}.
\end{enumerate}

The not-all-tied property in (i) follows immediately from Lemma~A.5 in \cite{BucSeg18}: note that $k_n$ in Condition~3.3 in that paper corresponds to $m_n$ here.

Consider (ii). Choose $\eta\in(2/\omega,\nu)$ with $\omega$ and $\nu$ from Condition~\ref{cond:alpha} and \ref{cond:moment}. Further, let  $0<\alpha_-<\alpha_0<\alpha_+$ be arbitrary (further constraints on  $\alpha_+$ will be imposed below).  
Lemma~A.6 in \cite{BucSeg18}  implies that $\lim_{n\to\infty}\Exp[\PP_nf] = Pf$, as long as $\alpha_+$ is chosen smaller than $2\alpha_0$ (in that case, any $f\in \Fc_2(\alpha_-, \alpha_+)$ can be bounded in absolute value by $g_{0,\alpha_0}$). Further, $\PP_nf - \Exp[\PP_nf]= m^{-1/2} \GG_n f=O_\Prob(m^{-1/2})=o_\Prob(1)$, as will be shown below in the proof of (iii). These two facts imply (ii). 

Consider (iii).  
The empirical process $\GG_n$ can be decomposed into a stochastic term and a bias term:
\[
\GG_n = \sqrt{m} (\PP_n - P_n) + \sqrt{m} (P_n - P) \equiv \tilde \GG_n  + B_n,
\]
where $P_n$ denotes the distribution of $X_{n,1}/\sigma_{r}$. For $j=1,2,3$, we have $B_n(f_j) \to B(f_j)$ by Condition~\ref{cond:bias}.
 Let us  show that the finite-dimensional distributions of $(\tilde \GG_n(f))_{f\in \Fc_2(\alpha_-, \alpha_+)}$ converge weakly to those of a zero-mean Gaussian process $\GG$ with covariance 
\[
\Cov(\GG f, \GG g)= 2\int_0^1 \Cov_{Q_{\alpha_0,\xi}} (f(U_1), g(U_2)) \, \diff \xi, \qquad f,g \in \Fc_2(\alpha_-, \alpha_+).
\]
This certainly implies (iii), and is also sufficient to close the missing gap in the proof of (ii) above.

By the Cram\'er--Wold device, it suffices to show weak convergence $\tilde \GG_n(h) \dto \GG(h)$ for  $h=v^Tg$ where $v$ is a column vector and where $g$ is a column vector of functions in $\Fc_2(\alpha_-, \alpha_+)$. Here, $\tilde \GG_n(h)$ and $\GG(h)$ are defined by linearity; in particular, $\GG(h)$ is centred Gaussian with variance $2v^T\big( \int_0^1 \Cov_{Q_{\alpha_0,\xi}} (g(U_1), g(U_2)) \, \diff \xi \big) v$, the integral being defined entrywise. Also, note that $|h|^{2+\delta} \lesssim g_{\eta,\alpha_0}$, provided we choose $\delta\in(2/\omega,\eta)$ and $\alpha_+>\alpha_0$ sufficiently small. Now, let 
\[
I_1=\{1, \dots, r\}, \quad I_2=\{r+1, \dots, 2r\}, \quad \dots, \quad I_m = \{(m-1)r+1, \dots, mr\}
\]
denote the indices making up the $m$ disjoint blocks of size $r$. Further, let $m^*=m^*_n \ge 3$, $m^*\le m$, be an integer sequence converging to infinity such that $m^*=o(m^{\delta/(2(1+\delta))})$ as $n\to\infty$. Define
\begin{align*}
J_1^+ &= I_1 \cup \dots \cup I_{m^*-2}, \quad & J_1^- &= I_{m^*-1} \cup I_{m^*},  \\
J_2^+ &= I_{m^*+1} \cup \dots \cup I_{2m^*-2}, \quad & J_2^- &= I_{2m^*-1} \cup I_{2m^*},  \\
J_3^+ &= I_{2m^*+1} \cup \dots \cup I_{3m^*-2}, &  J_3^- &= I_{3m^*-1} \cup I_{3m^*}, \qquad \dots 
\end{align*}
and so on, that is, successively merge $m^*-2$ of the initial disjoint blocks to a new big block $J_j^+$, and then 2 of the initial disjoint blocks to a new small block $J_j^-$. In total, we obtain $q=q_n=\ip{m/m^*} \to \infty$ new (disjoint) big blocks and small blocks. For simplicity, assume that $m^*q=m$, so that each time point $1, \dots ,n$  is covered by exactly one of the new blocks; otherwise, a negligible remainder term arises. We may then write
\begin{align*}
\tilde \GG_n(h) = \sqrt{m} \bigg( \frac{1}{k} \sum_{t=1}^{k} h(X_{n,t}/\sigma_{r})  - \Exp[h(X_{n,t}/\sigma_{r})] \bigg) 
&=
 \frac{1}{\sqrt{q}} \sum_{j=1}^{q} S_{nj}^{+} +   \frac{1}{\sqrt{q}}\sum_{j=1}^{q} S_{nj}^{-}  \\
 &\equiv A_n^++A_n^- , 
\end{align*}
where, for $j=1, \dots, q$,
\[
  S_{nj}^{\pm}
  = 
  \frac{\sqrt{mq}}{k} 
  \sum_{s \in J_j^{\pm}} 
  \{ h(X_{n,s}/\sigma_{r}) - \Exp[h(X_{n,s}/\sigma_{r})] \}. 
\]
It suffices to show that $A_n^- =o_\Prob(1)$ and that $A_n^+ \dto \GG(h)$ as $n\to\infty$. 

Let us prove that $A_n^-$ is negligible, and for that purpose consider its variance, since it is already centered. We have 
\begin{align*}
\Var(A_n^-) 
&= 
\Var(S_{n1}^-) + \frac{2}{q} \sum_{j=1}^{q-1} (q-j) \Cov(S_{n1}^-, S_{n,1+j}^-) \\
&\le
3\Var(S_{n1}^-) + 2 \sum_{j=2}^{q-1}  \Cov(S_{n1}^-, S_{n,1+j}^-).
\end{align*}
Recall that $|h|^{2+\delta} \lesssim g_{\eta,\alpha_0}$ with $\delta\in(2/\omega,\eta)$. Hence, by Condition~\ref{cond:moment} and  since  $q=\ip{m/m^*}$, 
\[
\| S_{n1}^-\|_{2+\eta} \le \frac{\sqrt{mq}}{k} 4r \cdot \| h(X_{n,1}/\sigma_{r}) \|_{2+\eta} \lesssim \frac1{\sqrt{m^*}} \cdot \| h(X_{n,1}/\sigma_{r}) \|_{2+\eta} =o(1).
\]
Thus $\Var(S_{n1}^-)=o(1)$ as well. Further, by Lemma~3.1 in \cite{DehPhi02},
\[
\sum_{j=2}^{q} \Cov(S_{n1}^-, S_{n,1+j}^-) 
\, \lesssim \,
q \, \lVert S_{n1}^- \rVert_{2+\delta}^2 \, \alpha(r)^{\delta/(2+\delta)} 
\, \lesssim \,
\frac{m}{(m^*)^2}  \alpha(r)^{\delta/(2+\delta)} ,
\]
which is of the order $o((m^*)^{-2})$ by Condition~\ref{cond:alpha} and the fact that $2/\delta>\omega$ by the choice of $\delta$. 

Now, consider the weak convergence of $A_n^+$. By a standard argument based on characteristic functions (see, e.g., the proof of Theorem 3.6 in \citealp{BucSeg18}), we may assume that the triangular array $S_{n1}^+, \dots, S_{nq}^+$ is rowwise independent. As a consequence, we may apply Lyapounov's central limit theorem (Theorem 27.3 in \citealp{Bil79}): 
provided that
$\Exp[ (S_{nj}^+)^2]$ converges to $\Var(\GG(h))$ and that
\begin{align} \label{eq:lya}
\lim_{n\to\infty}\frac{\sum_{j=1}^q \Exp[|S_{nj}^+|^{2+\delta}]}{\big(\sum_{j=1}^q \Exp[|S_{nj}^+|^2] \big)^{1+\delta/2}}=0,
\end{align}
we obtain that $A_n^{+}$ converges weakly to $\GG(h)$ and the proof of the claimed convergence of the finite-dimensional distributions of the empirical process $(\tilde \GG_n(f))_{f\in \Fc_2(\alpha_-, \alpha_+)}$  is finished. 
Now, 
\[
\| S_{nj}^+\|_{2+\delta} \le \frac{\sqrt{mq}}{k} 2 (m^*-2)r \cdot \| h(X_{n,1}/\sigma_{r}) \|_{2+\delta} \lesssim \sqrt{m^*} \cdot \| h(X_{n,1}/\sigma_{r}) \|_{2+\delta} =O(\sqrt{m^*})
\]
by Condition~\ref{cond:moment} (recall that $|h|^{2+\delta} \lesssim g_{\eta,\alpha_0}$).
As a consequence, provided that $\Exp[ (S_{nj}^+)^2]$ is converging, the fraction in \eqref{eq:lya} is of the order $O(q^{-\delta/2}(m^*)^{1+\delta/2})$. Since $q=\ip{m/m^*}$ and 
$m^*=o(m^{\delta/(2(1+\delta))})$, this expression converges to $0$. 

It remains to show that  $\Exp[ (S_{nj}^+)^2]=\Var(S_{nj}^+)$ converges to $\Var(\GG(h))$. This follows similarly as in the proof of Lemma~\ref{lem:cov}: since $q=\ip{m/m^*}$ and $m=\ip{n/r}$, we may write
\[
S_{n1}^+ = \frac{\sqrt{mq}}{k} \sum_{j=1}^{m^*-2} C_j =  \bigg( \frac{1}{r\sqrt{m^*}} \sum_{j=1}^{m^*-2} C_j \bigg)(1+o(1)), \qquad n\to\infty,
\]
where $C_j = \sum_{s\in I_j}\{h(Z_{r,s}) - \Exp[h(Z_{r,s})]\}$. Now,
\[
\Var\bigg(\frac{1}{r\sqrt{m^*}} \sum_{j=1}^{m^*-2} C_h\bigg) = \frac{m^*-2}{m^*} \frac{1}{r^2} \Var(C_1) + \frac{2}{r^2} \sum_{j=1}^{m^*-3} (1-\tfrac{(2+j)}{m^*}) \Cov(C_1, C_{1+j}).
\]
Exactly as in the proof of Lemma~\ref{lem:cov}, the right-hand side can be seen to be equal to
$
\Cov(C_2, C_1+C_2+C_3) /r^2 + o(1),
$
which further can be seen to converge to $\Var(\GG(h))=2 v^T \Cov(g(U_1), g(U_2)) \, v$, as asserted.
\end{proof}

\begin{proof}[Proof of Theorem~\ref{theo:ml}]
We apply Theorem~\ref{theo:weak} and need to check Conditions~\ref{cond:max}--\ref{cond:bias} as well as the expression of the bias function.

The max-domain of attraction Condition~\ref{cond:max} follows from first-order regular variation of $-\log F$ in \eqref{eq:RV}, which is a consequence of second-order regular variation in Condition~\ref{eq:SV:2}. Condition~\ref{cond:div} on the smallest block maxima follows in the same way as the proof of Condition~3.2 in the proof of Theorem~4.2 in \citet{BucSeg18}; in particular, \eqref{eq:ka} implies that $\log(m_n) = o(r_n)$ as $n \to \infty$, see Remark~4.5 in \citet{BucSeg18}. Strong mixing with rate in Condition~\ref{cond:alpha} is trivially fulfilled. Finally, Conditions~\ref{cond:moment} and~\ref{cond:bias} are the same as Conditions~3.4 and~3.5 in \citet{BucSeg18}, respectively, and are shown to be implied by second-order regular variation in the proof of Theorem~4.2 in the cited paper. \end{proof}

\begin{proof}[Proof of Corollary~\ref{cor:RT}]
The bias condition implies that
\begin{align*}
\sqrt{m_n} \bigg( \frac{\widehat \RL_n(T,r_n) }{\RL(T,r_n) } -1 \bigg) 
&=
\sqrt{m_n} \bigg( \frac{\widehat \RL_n(T,r_n) }{\widetilde \RL(T,r_n) } -1\bigg)  \frac{\widetilde \RL(T,r_n) }{\RL(T,r_n) } + \Lambda_n(T) \\
&= 
\sqrt{m_n} \bigg( \frac{\widehat \RL_n(T,r_n) }{\widetilde \RL(T,r_n) } -1 \bigg)(1+o(1)) + \Lambda_n(T).  
\end{align*}
The first factor on the right-hand side of this display can be written as
\begin{align*}
&\,
\sqrt{m_n} \Big( \frac{\hat \sigma_{n}}{ \sigma_{r_n} } b_T^{1/\alpha_0-1/\hat \alpha_n} - 1\Big) \\
=&\,
\sqrt{m_n}  \Big( \frac{\hat \sigma_{n}}{ \sigma_{r_n} }  - 1\Big)\big(b_T^{1/\alpha_0-1/\hat \alpha_n} - 1\big) + \sqrt{m_n} \Big( \frac{\hat \sigma_{n}}{ \sigma_{r_n} }  - 1\Big)+\sqrt{m_n} \big(b_T^{1/\alpha_0-1/\hat \alpha_n} - 1\big)  \\
=&\, 
o_\Prob(1) +\sqrt{m_n} \Big( \frac{\hat \sigma_{n}}{ \sigma_{r_n} }  - 1\Big)+\sqrt{m_n} (\hat \alpha_n - \alpha_0) \frac{\log(b_T)}{\hat \alpha_n \alpha_0} \\
=&\,
\beta(T,\alpha_0)' \cdot \sqrt{m_n} \,
  \begin{pmatrix} 
    \hat{\alpha}_n - \alpha_0  \\ 
    \hat{\sigma}_n / \sigma_{r_n} - 1  
  \end{pmatrix}  + o_\Prob(1),
\end{align*}
which proves the corollary.
\end{proof}

\appendix

\section{Covariance calculations for min-stable distributions}
\label{app:cov}

Let the pair of random variables $(S, T)$ have a min-stable distribution with unit exponential margins and Pickands dependence function $A : [0, 1] \to [1/2, 1]$, i.e.,
\[
  \Pr(S > x, T > y) = \exp\{ - (x+y) \, A(x/(x+y)) \}, \qquad (x, y) \in [0, \infty)^2 \setminus \{ (0, 0) \}.
\]
The function $A$ is convex and satisfies $w \vee (1-w) \le A(w) \le 1$ for all $w \in [0, 1]$. \cite{Tia80} obtained the formula
\begin{equation}
\label{eq:TdO}
  \Cov(\log S, \log T)
  =
  \int_0^1 \frac{ - \log A(w) }{w(1-w)} \diff w.
\end{equation}
To compute the asymptotic covariance matrix of the maximum likelihood estimator based on sliding blocks, we seek to generalize \eqref{eq:TdO} to 
\begin{equation}
\label{eq:Hklab}
  H_{k,\ell}(a,b) 
  = \Cov \bigl( S^a (\log S)^k, \, T^b (\log T)^\ell \bigr),
\end{equation}
for $a,b \in [0, \infty)$ and $k,\ell\in \N_0=\{0,1,2,\dots\}$. In particular, we are interested in the Marshall--Olkin distribution with parameter $\xi\in[0,1]$, which has Pickands dependence function $A = A_\xi$ in \eqref{eq:MO:A}.

For any $a \in [0, \infty)$ and $k\in\N_0$,
\begin{align*}
\Exp[S^a (\log(S))^k] =   \int_0^\infty s^a (\log(s))^k  \e^{-s} \, \diff s =  \Gamma^{(k)}(1+a),
\end{align*}
where $\Gamma^{(k)}$ denotes the $k$-th derivative of the Euler gamma function. For $k=1$, we may further write $\Gamma'(1+a) =  \Gamma(1+a) \, \psi(1+a)$, where $\psi = \Gamma'/\Gamma$ is the digamma function. For notational convenience, the zero-th (partial) derivative of a function is to be interpreted as the function itself.

\begin{lemma}
Let $(S, T)$ have a min-stable distribution with unit exponential margins and Pickands dependence function $A$. Then, for all $a,b>0$ and all $k,\ell\in\N = \{0,1,2,3,\dots\}$, we have
\begin{align*}
H_{k,\ell}(a,b) 
	&= \int_0^1 \frac{\partial^{k+\ell}}{\partial a^k \partial b^\ell} \left( ab \, \Gamma(a+b) \frac{w^{a-1} (1-w)^{b-1}}{(A(w))^{a+b}} \right) \, \diff w - 
	\Gamma^{(k)}(1+a) \, \Gamma^{(\ell)}(1+b), \\
H_{k,1}(a,0) 
	&= \int_0^1 \frac{\partial^k}{\partial a^k} \left( \Gamma(a+1) \big( (A(w))^{-a} - 1 \big) w^{a-1} (1-w)^{-1} \right) \, \diff w, \\
H_{1,1}(0,0) 
	&=\int_0^1 \frac{ - \log A(w) }{w(1-w)} \diff w.
\end{align*}
Further cases can be obtained by symmetry: we have $H_{k,\ell}(a,b)=\tilde H_{\ell,k}(b,a)$, where $\tilde H$ is given by the above formulas, but with $A(w)$ replaced by $A(1-w)$.
\end{lemma}

\begin{proof} In the following, let $a,b>0$. 

\medskip\noindent\emph{$\bullet$ Case $H_{0,0}(a,b)$.} By Fubini's theorem,
\begin{align*}
  \Exp[S^a T^b]
  &=
  \Exp \left[ \int_0^S ax^{a-1} \, \diff x \, \int_0^T by^{b-1} \diff y \right] \\
  &=
  \Exp \left[ \int_{(0, \infty)^2} \ind(S > x, T > y) \, ax^{a-1} \, by^{b-1} \diff (x, y) \right] \\
  &=
  \int_{(0, \infty)^2} \Prob(S > x, T > y) \, ax^{a-1} \, by^{b-1} \diff (x, y) \\
  &=
  \int_{(0, \infty)^2} \e^{-(x+y) \, A(x/(x+y)) } \, ax^{a-1} \, by^{b-1} \diff (x, y).
\end{align*}
Substituting $x+y = s$ and $x/(x+y) = w$, so $x = sw$ and $y = s(1-w)$, with Jacobian $\lvert \partial(x, y) / \partial(s, w) \rvert = s$, and following up by substituting $s \, A(w) = t$, we find
\begin{align*}
  \Exp[S^a T^b]
  &=
  \int_{w=0}^1 \int_{s=0}^\infty \e^{-s \, A(w)} \, a(sw)^{a-1} \, b(s(1-w))^{b-1} \, s \diff s \diff w \\
  &=
  ab\int_{w=0}^1 \int_{t=0}^\infty \e^{-t} \, (tw/A(w))^{a-1} \, (t(1-w)/A(w))^{b-1} \, (A(w))^{-2} \, t \, \diff t \diff w \\
  &=
  ab\int_{w=0}^1 \frac{w^{a-1} (1-w)^{b-1}}{(A(w))^{a+b}} \int_{t=0}^\infty t^{a+b-1} \e^{-t} \diff t \diff w \\
  &=
  ab \, \Gamma(a+b) \int_{w=0}^1 \frac{w^{a-1} (1-w)^{b-1}}{(A(w))^{a+b}} \, \diff w.
\end{align*}

\medskip\noindent\emph{$\bullet$ Case $H_{k,\ell}(a,b)$.} 
By the mean value theorem, we have, for $b > 0$, $t > 0$ and $h$ such that $b+h > 0$, the inequality
\[
  \left\lvert \frac{t^{b+h} - t^b}{h} \right\rvert
  \le
  \max(t^{b+h}, t^b) \lvert \log t \rvert.
\]
Since $\Exp[ T^b (\log(T))^\ell ] < \infty$ for all $b > 0$ and $\ell = 0, 1, 2, \ldots$, an application of the dominated convergence theorem implies that we can interchange expectation and partial derivatives to find that
\begin{align*}
  \Exp[ S^a (\log(S))^k T^b (\log(T))^\ell ]
  &=
  \Exp\left[ 
    \frac{\partial^k S^a}{\partial a^k} 
    \frac{\partial^\ell T^b}{\partial b^\ell} 
  \right]
  =
  \frac{\partial^{k+\ell}}{\partial a^{k} \partial b^\ell} 
  \Exp[ S^a T^b ] \\
  &=
  \frac{\partial^{k+\ell}}{\partial a^{k} \partial b^\ell} 
  \int_{w=0}^1 
    ab \, \Gamma(a+b)  
    \frac{w^{a-1} (1-w)^{b-1}}{(A(w))^{a+b}} \, 
  \diff w.
\end{align*}
By the same type of argument, we can interchange the partial derivatives and the integral over $w$: in a small neighbourhood of a fixed pair $(a, b) \in (0, \infty)^2$, the partial derivatives with respect to $a$ and $b$ are bounded by a constant multiple of the function $w \mapsto w^{a-h-1} \lvert \log(w) \rvert^{k} w^{b-h-1} \lvert \log(1-w) \rvert^{\ell}$, for some $0 < h < a \wedge b$, a function which is integrable over $w \in (0, 1)$.

\medskip\noindent\emph{$\bullet$ Case $H_{0,1}(a,0)$.} By Hoeffding's covariance formula, 
\begin{align*}
\Cov\left( S^a, \log(T) \right)
=
\int_{(0,\infty)\times\R} \{ \Prob( S^a > x, \log(T) >y) - \Prob( S^a > x) \, \Prob(\log(T) >y) \} \, \diff (x,y),
\end{align*}
where 
\[
\Prob( S^a > x, \log(T) >y) = \exp\left\{ - (x^{1/a}+\e^{y}) A \left( \frac{x^{1/a}}{x^{1/a}+\e^{y}} \right) \right\}.
\]
Apply the change of variables $s=x^{1/a}+\e^{y}$ and $w=x^{1/a}/(x^{1/a}+\e^{y})$, so $x=(sw)^a$ and $y=\log(s(1-w))$ with Jacobian $\lvert \partial(x, y) / \partial(s, w) \rvert = as^{a-1}w^{a-1}(1-w)^{-1}$. We obtain
\begin{align*}
\Cov\left( S^a, \log(T) \right)
&=
  \int_{0}^1 
    \left\{ 
      \int_0^\infty 
	\bigl( \e^{-sA(w)} - \e^{-s} \bigr) a s^{a-1} \, 
      \diff s 
    \right\} 
    w^{a-1}(1-w)^{-1} \, 
  \diff w \\
&=
  \int_{0}^1 
    \bigl( (A(w))^{-a} - 1 \bigr)
    \Gamma(a+1) \,
    w^{a-1}(1-w)^{-1} \, 
  \diff w.
\end{align*}

\medskip\noindent\emph{$\bullet$ Case $H_{k,1}(a,0)$.}
As in the case $H_{k,\ell}(a, b)$, we can interchange expectation (or integration) and differentiation to find
\begin{align*}
\Cov\big( S^a \log(S), \log(T) \big) 
&=
\Cov\left( \frac{\partial^k S^a}{\partial a^k} , \log(T) \right) \\
&=
\frac{\partial^k }{\partial a^k}  \Cov\left( S^a, \log(T) \right) \\
&=
\int_0^1 
  \frac{\partial^k}{\partial a^k } 
  \left\{ 
    \Gamma(a+1) \, 
    \bigl( (A(w))^{-a} - 1 \bigr) \,
    w^{a-1} (1-w)^{-1} 
  \right\} \, 
\diff w.
\end{align*}

\medskip\noindent\emph{$\bullet$ Case $H_{1,1}(0,0)$.} This is \eqref{eq:TdO} and can be found by a similar argument as the one for $H_{0, 1}(a, 0)$.
\end{proof}

Of special interest are the cases $a=b=1$ and $k,\ell \in\{0,1\}$.

\begin{corollary}
\label{cor:cov:A}
Let $(S, T)$ have a min-stable distribution with unit exponential margins and Pickands dependence function $A$. With $H_{k,\ell}(a, b)$ as in \eqref{eq:Hklab} we have
\begin{align*}
H_{0,0}(1,1) 
	&=  \int_{0}^1 \frac{1}{(A(w))^{2}} \, \diff w - 1, \\ 
H_{0,1}(1,1) 
	&= \int_0^1 \frac{1}{(A(w))^2} \Big[ 1 + \log(1-w) + \psi(2) - \log(A(w)) \Big] \, \diff w - \psi(2), \\ 
H_{1,1}(1,1) 
	&=  \int_0^1 \frac{1}{(A(w))^2} \Big[ \psi(2)^2 + 2 \psi(2) + \psi'(2) + 1  \\
  	& \hspace{2cm} + (1+\psi(2)) \big\{ \log(w) + \log(1-w) - 2 \log(A(w)) \big\} \\
	& \hspace{2cm} + \big( \log(w) -  \log(A(w)) \big) \big( \log(1-w) -  \log(A(w)) \big)  \Big] \, \diff w  - \psi(2)^2, \\ 
H_{0,1}(1,0) 
	&=  \int_{0}^1  \frac{1-A(w)}{(1-w) A(w)} \, \diff w, \\ 
H_{1,1}(1,0) 
	&= \int_0^1    \frac{\big(1-A(w)\big) \big( \log(w) + \psi(2) \big) - \log(A(w))  }{(1-w)A(w)}  \, \diff w, \\ 
H_{1,1}(0,0) 
	&=\int_0^1 \frac{ - \log A(w) }{w(1-w)} \diff w. 
\end{align*}
\end{corollary}

Of further special interest is the Marshall--Olkin Pickands dependence function $A_\xi$ in~\eqref{eq:MO:A}. Write $H_{k,\ell}(a, b; \xi)$ for the covariance in \eqref{eq:Hklab} if $A = A_\xi$ with $\xi \in [0, 1]$. 
%
%
The asymptotic covariance matrix of the maximum likelihood estimator based on sliding blocks involves the integrals of $H_{k,\ell}(a, b; \xi)$ over $\xi \in [0, 1]$.

\begin{corollary} 
\label{cor:cov:int}
We have
\begin{align*}
\int_0^1 H_{0,0}(1,1;\xi) \, \diff \xi
	&=  2\log(2)-1, \\ 
\int_0^1 H_{0,1}(1,1;\xi) \, \diff \xi
	&=  \frac{\pi^2}{12} - \log^2(2) + (1-\gamma)( 2\log(2) -1),  \\ 
\int_0^1 H_{1,1}(1,1;\xi) \, \diff \xi
	&= 2 \log(2) \Big( \psi(2)^2 + \frac{\pi^2}{6} - \psi(2) \log(2) + \log^2(2)/3 \Big) \\
	&\qquad \mbox{} + \psi(2) \frac{\pi^2}{6} - \frac{7}{4}\zeta(3) - \psi(2)^2, \\
\int_0^1 H_{0,1}(1,0;\xi) \, \diff \xi
	&=   \frac{\pi^2}{12}  + 1-2\log(2), \\ 
\int_0^1 H_{1,1}(1,0;\xi) \, \diff \xi
	&= (1+\psi(2))\frac{\pi^2}{12}+\log^2(2)-2\psi(2)\log(2)+\psi(2)-\frac{7}{8}\zeta(3),  \\
\int_0^1 H_{1,1}(0,0;\xi) \, \diff \xi
	&= 4 \log(2) -2.   
\end{align*}
\end{corollary}

\begin{proof}
The expressions of $H_{k,\ell}(a, b; \xi)$ follow from Corollary~\ref{cor:cov:A} with $A = A_\xi$. Integrating over $\xi$ yields a double integral over $\xi$ and $w$. By interchanging integration with respect to $\xi$ and $w$ if necessary, we can calculate all six integrals explicitly. The algebraic details are tedious and are omitted for brevity. All formulas have been checked numerically. 
\end{proof}

\section*{Acknowledgments}

The authors would like to thank three anonymous referees and an Associate Editor for their constructive comments on an earlier version of this manuscript. Axel Bücher gratefully acknowledges support by the Collaborative Research Center ``Statistical modeling of nonlinear dynamic processes'' (SFB 823) of the German Research Foundation. Johan Segers gratefully acknowledges funding by contract ``Projet d'Act\-ions de Re\-cher\-che Concert\'ees'' No.\ 12/17-045 of the ``Communaut\'e fran\c{c}aise de Belgique'' and by IAP research network Grant P7/06 of the Belgian government.

\end{document}